\documentclass[journal]{IEEEtran}
\ifCLASSINFOpdf
   \usepackage[pdftex]{graphicx}
\else
   \usepackage[dvips]{graphicx}
\fi
\usepackage[utf8]{inputenc}
\renewcommand{\l}{\left}
\renewcommand{\r}{\right}

\usepackage{enumerate}
\usepackage{amsmath}
\usepackage{amssymb}
\usepackage{amsthm}
\usepackage{xspace}
\usepackage{comment}
\usepackage{algorithm}
\usepackage{algorithmic}
\usepackage{eufrak}
\usepackage{color}
\definecolor{red}{rgb}{1.0,0.0,0.0}
\definecolor{blue}{rgb}{0.0,0.0,1.0}
\newcommand{\niek}[1]{\textcolor{red}{\textsf{Niek: #1}}\xspace}
\newcommand{\andy}[1]{\textcolor{blue}{\textsf{Andrey: #1}}\xspace}

\newcommand{\reflem}[1]{Lemma~\ref{lem:#1}\xspace}
\newcommand{\refthm}[1]{Theorem~\ref{thm:#1}\xspace}
\newcommand{\refdef}[1]{Definition~\ref{def:#1}\xspace}
\newcommand{\refcor}[1]{Corollary~\ref{cor:#1}\xspace}

\newcommand{\refsec}[1]{Section~\ref{sec:#1}\xspace}

\newcommand{\reffig}[1]{Figure~\ref{fig:#1}\xspace}
\newcommand{\refeq}[1]{Eqn.~\eqref{eq:#1}\xspace}
\newcommand{\refprop}[1]{Proposition~\ref{prop:#1}\xspace}

\newcommand{\defas}{:=} 
\newcommand{\term}[1]{\textsl{#1}\xspace}

\newcommand{\cl}{\textsc{commelec}\xspace}
\newcommand{\pqprof}{$PQ$ profile\xspace}
\newcommand{\posint}{\ensuremath{\mathbb{N}_{>0}}\xspace}
\newcommand{\mcal}[1]{\ensuremath{\mathcal{#1}}\xspace}

\newtheorem{theorem}{Theorem}
\newtheorem{lemma}{Lemma}
\newtheorem{corollary}{Corollary}
\newtheorem{proposition}{Proposition}
\newtheorem{observation}{Observation}
\theoremstyle{definition}
\newtheorem{remark}{Remark}
\newtheorem{example}{Example}
\newtheorem{definition}{Definition}

\newcommand{\diam}{\mathsf{diam}}

\newcommand{\vsect}[2][c]{#2_{\{#1\}}} 
\newcommand{\conv}{\mathsf{ch}}

\newcommand{\vor}[2][]{\mathrm{Vor}_{#1}(#2)} 

\newcommand{\natnum}{\mathbb{N}}
\newcommand{\reals}{\mathbb{R}}

\renewcommand{\S}{\ensuremath{\mathcal{S}}\xspace}
\renewcommand{\P}{\ensuremath{\conv{\S}}\xspace}
\newcommand{\A}{\ensuremath{\mathcal{A}}\xspace}

\renewcommand{\SS}{\ensuremath{\mathbb{S}}\xspace}

\DeclareMathOperator*{\argmin}{arg\,min}

\ifCLASSOPTIONcompsoc
  \usepackage[caption=false,font=normalsize,labelfont=sf,textfont=sf]{subfig}
\else
  \usepackage[caption=false,font=footnotesize]{subfig}
\fi

\begin{document}
%
\title{Real-Time Minimization of Average Error in the Presence of Uncertainty and Convexification of Feasible Sets}

%
%
%

\author{
  Andrey~Bernstein,~\IEEEmembership{Member,~IEEE,}
  Niek~J.~Bouman,~\IEEEmembership{Member,~IEEE,}
  Jean-Yves Le Boudec,~\IEEEmembership{Fellow,~IEEE}
\thanks{The authors are with EPFL, Lausanne, Switzerland.}
\thanks{Manuscript received XXXXXXXX, 2015; revised XXXXXXX, 2015.}}

\maketitle

\begin{abstract}
  We consider a two-level discrete-time control framework with real-time constraints 
where a central controller issues setpoints to be implemented by local controllers. The local controllers implement the setpoints with some approximation and advertize a prediction of their constraints to the central controller. The local controllers might not be able to implement the setpoint exactly, due to prediction errors or because the central controller convexifies the problem for tractability. 
%
%
In this paper, we propose to compensate for these mismatches at the level of the local controller by 
using a variant of the error diffusion algorithm. 
We give conditions under which the minimal (convex) invariant set for the accumulated-error dynamics is bounded, and give a computational method 
to construct this set. 
This can be used to compute a bound on the accumulated error and hence establish convergence of the average error to zero. 
We illustrate the approach in the context of real-time control of electrical grids. 

\end{abstract}


%
\IEEEpeerreviewmaketitle

\section{Introduction}

We consider a two-level discrete-time control framework with real-time constraints, consisting of several \emph{local controllers} and a \emph{central controller}. 
Such a framework was recently proposed in the context of real-time control of electrical grids \cite{commelec}.
The task of a local controller is (i) to implement setpoints issued by the central controller, and (ii) to advertise a prediction of the constraints on the feasible setpoints to the central controller. In turn, the central controller uses these advertisements to compute next feasible setpoints for the local controllers. These setpoints correspond to the solution to an optimization problem posed by the central controller. Due to real-time constraints, the central controller is restricted to operate on \emph{convex feasible sets} and \emph{continuous variables}. Hence, the advertisement of the local controllers is in the form of convex sets.

  
  Formally, the interaction between the local controllers and the central controller is assumed to be as in Algorithm \ref{alg:inter}.
\begin{algorithm}[h!]
\caption{Interaction between local controllers and central controller} \label{alg:inter}
\begin{algorithmic}[1]
\STATE{Set $n = 0$.}
\LOOP
    \STATE{At time step $n$, every local controller:} \label{state:n}
    \begin{enumerate}[(a)]
        \item Receives a setpoint request $x_n$ sent by the central controller.
        \item Implements an approximation $y_n$ of $x_n$. The implemented setpoint $y_n$ is constrained to lie in some set $\S_n \in \SS$, where $\SS$ is the collection of all possible feasible sets of the local controller.
        \item Performs a prediction of its  feasible set $\S_{n+1} \in \SS$ that will be valid at step $n+1$, and advertises to the central controller the convex hull $\A_{n+1}=\conv \S_{n+1}$.
    \end{enumerate}
    \STATE{Upon receiving the advertisements from all the local controllers, the central controller chooses the setpoints $x_{n+1} \in \A_{n+1}$ for every local controller and sends them over a communication network.}
    \STATE{$n := n + 1$.}
\ENDLOOP
\end{algorithmic}
\end{algorithm}

  \begin{example}
  Consider a central grid controller, whose task is to control the grid and the resources connected to it in real-time.
  Consider a single-phase (or balanced) system, where the setpoints are pairs $(P, Q) \in \reals^2$ that represent the requested active ($P$) and reactive ($Q$) power consumption/production. As an example of the resources, consider a photovoltaic (PV) plant and a heater system with a \emph{finite} number of heating states. In the case of the PV, the local controller predicts the set of feasible (implementable) power setpoints $\S_{n+1}$ in step 3(c) of Algorithm \ref{alg:inter}, but the actual set at the time of the implementation may differ from the prediction due to high volatility of solar radiation. In the case of the heater system, the set $\S_{n+1}$ is a finite set. Thus, the corresponding local controller sends a convex hull of $\S_{n+1}$ in step 3(c) of Algorithm \ref{alg:inter}, and consequently may receive a non-implementable setpoint from the central controller in the next step. 
  \end{example}

  This example illustrates the source of the potential difference between $y_n$ and $x_n$ in step 3(b) of Algorithm \ref{alg:inter}. More generally, the local controller might not be able to exactly implement the requested setpoint, because of two reasons: 
  \begin{enumerate}
  \item \emph{Due to convexification of the feasible set.} The actual set of implementable setpoints might be non-convex. For example, suppose that the local controller is controlling a collection of ``on-off'' devices, which would correspond to a discrete set of implementable setpoints.
 \item\emph{Because of uncertainty.} The actual set of implementable setpoints might differ from its prediction, for example, because of external disturbances.
  \end{enumerate}
  
  In this paper, we propose to use the metric of total accumulated error between  requested and implemented setpoints.
  Our goal is to analyze the \emph{greedy algorithm} for the choice of $y_n$ in step 3(b) of Algorithm \ref{alg:inter}, namely the algorithm that  performs \emph{online minimization of the accumulated error}. (This algorithm is also called the \emph{error diffusion} algorithm and is well-known in the context of image processing and digital printing.) We show conditions on the collection $\SS$ under which the accumulated error is bounded for all $n$ and propose computational methods to find tight bounds. As a consequence, the average error converges to zero. In Section \ref{sec:rel_work}, we discuss how our approach compares to the existing literature. 
  
  In general, the boundedness of accumulated error (and thus convergence of the average error to zero) is a relevant metric in any application where the integral of the control variable is an important quantity. In this paper, we illustrate our approach in the context of real-time control of electrical grids \cite{commelec}, where the setpoints are active and reactive power injections/absorptions, and the integral thereof is the consumed/produced energy. 
  In this application, ``real-time'' means a period time in the order of $100$ ms. Thus, the proposed framework with a fast and efficient central controller is a natural solution, as discussed in detail in \cite{commelec}. Other applications in which boundedness of accumulated error is relevant include signal processing, digital printing, scheduling, and assignment problems \cite{adler2005, errDiffOpt}.

\subsection{Related Work} \label{sec:rel_work}

The problem of the local controller can be viewed as controlling the quantity $e_n$ defined recursively by
\[
e_{n+1} = e_n + x_n - y_n,
\]
where $x_n$ is the requested setpoint by the central controller and $y_n$ is an implemented setpoint by the local controller; see Algorithm \ref{alg:inter}.
Note that in this context, $y_n$ is an \emph{input} (or \emph{control}) variable, while $x_n$ can be viewed as an exogenous disturbance variable. 

The task of the local controller is  to find a policy that achieves $e_n \in Q$ for all $n$, for some bounded set $Q$. At first glance, this is a standard goal in control theory, and involves a \emph{robust control invariant set} \cite{bla99}. 
Moreover, when the set $\S_n$ is non-convex or  discrete, this problem is closely related to control of Mixed Logic Dynamical (MLD) Systems (e.g., \cite{Branicky, Bemporad, LYGEROS}). The classical approach to design a robust controller in this setting is based on formulating a corresponding MPC and solving MILP or MIQP problems.
However, in real-time applications, solving MIQP/MILP at every time step may be not feasible.
Moreover, our problem has the following features which are not present in the classical setting:
\begin{enumerate}
\item The feasible set $\S_n$ is not known in advance and depends on the history of the process up to time step $n$.
\item The feasible set $\S_n$ can be uncertain in the sense that it is not known to the central controller at the time of the decision making, and thus prediction errors might arise.
\item We allow to control the set of possible ``disturbances'' $\A_n$ -- it is the advertisement sent by the local controller.  
\end{enumerate}
Thus, the methods from the classical control theory do not apply directly.

A different approach, which we pursue in the present paper, is to directly analyze a specific algorithm for the choice of $y_n$, namely a \emph{greedy algorithm} which, at each time step, minimizes the next-step accumulated error. In particular, it chooses $y_n = \vor[\S_n]{x_n + e_{n}}$, where $\vor[\S]{z}$ is the closest point to $z$ in $\S$. The classical version of this algorithm is known as \emph{error diffusion} (or \emph{Floyd-Steinberg dithering}) in the field of image processing and digital printing, where the variables are one-dimensional \cite{floyd75, gray, Anastassiou, adams}. 
The extension to the general $d$-dimensional case was considered over the recent years in several papers.
In \cite{nowicki2004}, the problem of a single (fixed) feasible set $\S$ is considered in the special case where $\S$ are the \emph{corner points} of a polytope $\A = \conv{\S}$, and an algorithm to construct a minimal invariant set for the corresponding dynamical system is proposed. However, the boundedness of this set is not guaranteed in general.
\cite{adler2005} show how to construct bounded invariant sets for that problem, and extend the results to a finite collection of polytopes $\A_1,\ldots,\A_K$. Namely, they show that there exists a bounded set  $Q$ that is simultaneously invariant for the dynamical systems defined with respect to $\A_1,\ldots,\A_K$.
\cite{tresser2007} extends the results of \cite{adler2005} to the case where the polytope may change from step to step, and argues that there exists a bounded invariant set $Q$ for this changing dynamical system. Moreover, the conditions are extended to an infinite collection of polytopes, provided that the set of face-normals of this collection is finite.
Other papers in this line of research consider specific applications \cite{boundsHalftoning, mathHalftoning, errDiffOpt} and/or other special cases \cite{acute}. Finally, the optimality of the error diffusion algorithm was recently analyzed in \cite{errDiffOpt}.

\subsection{Our Contribution}
Our paper extends the state-of-the-art on the general error-diffusion algorithm, with the following main contributions that are relevant to our control application:
\begin{itemize}
\item We consider general non-convex feasible sets $\S$ rather than corner points of a polytope.
\item We consider an uncertain case, in which the feasible set is not known at the time of advertisement and hence predicted. Specifically, in this case, it may happen that $\A \neq \conv{\S}$.
\item We propose 
a computational method for constructing the minimal invariant set for the accumulated error dynamics in the case of finite collection of feasible sets. We also show some important special cases in which  the minimal invariant set can be computed explicitly. As a result, we obtain tight bounds on the accumulated error.
\end{itemize}

\section{Notation}

Throughout the paper, $\|\cdot\|$ denotes the $\ell_2$ norm. We use $\natnum$ and  $\posint$ to refer to the natural numbers including and excluding zero, respectively. For arbitrary $n\in \natnum_{>0}$, we write $[n]$ for the set $\{1,\ldots,n\}$. 

Let $d\in \natnum_{>0}$.
For arbitrary sets $\mcal{U},\mcal{V} \in \reals^d$,  $\mcal{U}+\mcal{V}$ represents the Minkowski sum of \mcal{U} and \mcal{V}, which is defined as
$\mcal{U}+\mcal{V} := \{u+v \,|\,  u \in \mcal{U},  v\in \mcal{V} \}$. Likewise, $\mcal{U}-\mcal{V}$ represents the Minkowski difference, defined as $\mcal{U}-\mcal{V} := \{u -v) \,|\,  u \in \mcal{U},  v\in \mcal{V} \}$. We let $\conv \mcal{U}$ denote the convex hull of the set $\mcal{U}$, and by $\partial \conv \mcal{U}$ we denote the \emph{boundary} of the convex hull of \mcal{U}.

For any compact set \mcal{V}, we define the
\term{diameter of \mcal{V}} as
\[
\diam \mcal{V} := \max \{ \| v-w \| : v,w \in \mcal{V} \}.
\]

  Let $\mcal{S} \subset \reals^d$ be an arbitrary non-empty closed set.
  Any mapping $\vor[\mcal{S}]{x}$ that satisfies
\[
  \vor[\mcal{S}]{x} =  c, \quad\text{where}\quad
c \in \arg \min_{\rho \in\mcal{S}} \|\rho-x \|.
\]
is called  a \term{closest-point} (or \emph{projection}) operator onto \mcal{S}.
The \emph{Voronoi cell} associated with the set $\S$ and a point $c \in \S$ is defined as
\[
V_{\S}(c) = \{x \in \reals^d: \, \|x - c\| \leq \|x - c'\|, \forall c' \in \S\}.
\]
For any set $\mcal{V} \subseteq \reals^d$, we denote the intersection of $\mcal{V}$ with $V_{\S}(c)$ by
\[
\mcal{V} \cap V_{\S}(c) := \mcal{V}_{\{c\}}
\]
whenever the set $\S$ is clear from the context.

%
%
%





Finally, throughout the paper, a ``set'' (or ``subset'') denotes a ``closed set'' (or ``closed subset'') unless specified otherwise explicitly.

\section{Problem Definition}
Fix the dimension $d\in \natnum$. Let $\SS$ be a collection of subsets of $\reals^d$. This collection represents the set of all possible feasible sets of the local controller. Recall that the interaction between the local controllers and the central controller is given in Algorithm \ref{alg:inter}.

  In this paper, we focus on two cases for the prediction step 3(c):
  \begin{enumerate} [(i)]
      \item \emph{Perfect prediction}, namely $\A_{n+1} = \conv \S_{n+1}$, and
      \item \emph{Persistent prediction}, namely $\A_{n+1} = \conv \S_{n}$.
  \end{enumerate}
  
  The performance metric considered is the \emph{accumulated error} defined recursively by
  \begin{equation} \label{eqn:e_n}
  e_{n+1} = e_n + x_n - y_n.    
  \end{equation}

  We analyze the \emph{greedy algorithm} for the choice of $y_n$, namely the algorithm that chooses $y_n \in \S_n$ so that $\|e_{n+1}\|$ is minimized. That is,
  \begin{equation} \label{eqn:err_diff}
  y_n = \argmin_{y \in \S_n} \left\| e_n + x_n - y \right\| = \vor[\S_n]{e_n + x_n}
  \end{equation}
   which is the closest point to $e_n + x_n$ in $\S_n$. This algorithm is also known as \emph{error diffusion}.
  
  As mentioned in the introduction, our goal in this paper is to:
  \begin{enumerate}
      \item Find conditions on the collection $\SS$ under which the accumulated error $e_n$ is bounded for all $n$.
      \item Propose computational methods to find tight bounds.
  \end{enumerate}

\section{Main Results}
In this section, we present our main results for the two cases in the prediction step 3(c) of Algorithm \ref{alg:inter}. The proofs are deferred to Section \ref{sec:proofs}.

\subsection{Perfect Prediction} \label{sec:perf}
Recall from Algorithm \ref{alg:inter} that since $\A_{n} = \conv \S_{n}$, the setpoint request $x_n$ lies in $\conv \S_{n}$, while the local controller implements a setpoint $y_n$ according to \eqref{eqn:err_diff}. Therefore,  the dynamics for the accumulated error variable $e_n$ \eqref{eqn:e_n} is given by
\begin{equation} \label{eqn:e_n_perfect}
e_{n+1} = e_n + x_n - \vor[\S_n]{e_n + x_n}.
\end{equation}
Similarly to \cite{adler2005, tresser2007}, 
for any non-empty 
set $\S \subset \reals^d$ and any $x \in \conv{\S}$ we define the map
\begin{align*}
G_{\S,x }:  \reals^d &\rightarrow \reals^d  \\
e & \mapsto e + x - \vor[\S]{e + x}.
\end{align*}
The dynamics for the accumulated error \eqref{eqn:e_n_perfect} can be then expressed as
  \begin{equation} \label{eqn:err_diff_dyn}
  e_{n+1} = G_{\S_n,x_n }(e_n).    
  \end{equation}

\begin{definition}[Invariance] \label{def:invariance}
We say that a set $Q \subseteq \reals^d$ is $G$-invariant with respect to a 
set $\S \subset \reals^d$
if
\[
\forall \, x \in \P, \quad G_{\S,x}(Q) \subseteq Q.
\]
We say that $Q$ is $G$-invariant with respect to a collection $\SS$ if it is $G$-invariant with respect to every $\S \in \SS$. 
\end{definition}

\begin{remark}
Our definition of invariance is a special case of
\emph{robust positively invariant sets}  in robust set-invariance theory \cite{bla99}.
Indeed, for any dynamical system $s_{n+1} = f(s_n, w_n)$, a robust positively invariant set is any set $\Omega$ that satisfies $f(\Omega, w) \subseteq \Omega$ for all $w \in \mathbb{W}$, where $\mathbb{W}$ is a set of all possible disturbances. In our case, we can view the pair $(\S_n, x_n)$ as a disturbance $w_n$ that lies in the set
\[
\mathbb{W} := \{(\S, x): \, \S \in \SS, x \in \conv \S \}.
\]
\end{remark}

The following observation makes the connection between invariant sets and boundedness of the accumulated error.
\begin{observation} \label{prop:inv_bound}
Let $\SS$ be a collection of subsets of $\reals^d$. Consider the dynamics \eqref{eqn:err_diff_dyn}, where $\S_n \in \SS$ for all $n$. Let $Q$ be a $G$-invariant set with respect to the collection $\SS$. Then, if $e_0 \in Q$, it holds that $e_n \in Q$ for all $n \geq 1$.
\end{observation}
In particular, it follows from Observation \ref{prop:inv_bound} that if $Q$ is a bounded invariant set that contains the origin, and $e_0 \in Q$, then the accumulated error is bounded for all $n$ by $\max_{v \in Q} \|v\|$.

Our next goal is to find \emph{minimal} $G$-invariant sets in order to obtain \emph{tight} bounds for the accumulated error. In particular, we provide conditions for the existence of \emph{bounded} minimal $G$-invariant sets, and a method to compute these sets. To that end, we first define the following set-operators similarly to \cite{nowicki2004}.

\begin{definition}[Set operators induced by a collection of sets] \label{def:g}
Fix $d \in \natnum$.
For any 
set 
$\S \subseteq \reals^d$, define:
\[
\mathfrak{g}_\S(Q):=\bigcup_{c \in \S } \vsect{(\P+Q)} -c.
\]
For a collection $\SS$, let
\[
\mathfrak{g}_\SS(Q):=\bigcup_{\S \in \SS} \mathfrak{g}_\S(Q).
\]
\end{definition}

We have the following alternative characterization of invariance in terms of the operator $\mathfrak{g}_\SS$.

\begin{proposition} \label{prop:inv_mult}
\begin{enumerate}[(i)]
\item[]
\item For any $Q \subseteq \reals^d$, 
\[
Q \subseteq \mathfrak{g}_\SS(Q).
\]
\item $Q \subseteq \reals^d$ is $G$-invariant with respect to $\SS$ if and only if 
\[
\mathfrak{g}_\SS(Q) \subseteq Q .
\]
\end{enumerate}
\end{proposition}

The next two theorems show how to find minimal invariant sets and provide conditions on their boundedness.

\begin{theorem}[Minimal Invariant Set] \label{theo:min_G}
Let $Q \subseteq \reals^d$. The iterates
\[
\mathfrak{g}^n_\SS (Q) := \mathfrak{g}_\SS (\mathfrak{g}^{n-1}_\SS (Q)) \quad \text{for } n\in \natnum, n \geq 1, \quad \mathfrak{g}^0_\SS (Q) := Q,
\]
are monotonic, in the sense that $\mathfrak{g}^n_\SS (Q) \subseteq \mathfrak{g}^{n'}_\SS (Q)$ for all $n \leq n'$, and the limit set 
\[
\mathfrak{g}^\infty_\SS (Q) := \lim_{n \rightarrow \infty} \mathfrak{g}^n_\SS (Q) = \bigcup_{n\geq0} \mathfrak{g}^n_\SS(Q)
\]
is the \emph{minimal $G$-invariant set} with respect to $\SS$ that contains $Q$.
\end{theorem}

\begin{theorem}[Bounded Invariant Set] \label{theo:bound_G}
If the collection \SS is such that $\conv{\SS} := \{\conv{\S}, \S \in \SS \}$ is a collection of polytopes such that:
\begin{enumerate}[(i)]
    \item The sizes of the polytopes are uniformly bounded;
    \item The set $\mathcal{N}$ of outgoing normals to the faces of the polytopes is finite; and
    \item The bounded Voronoi cells of $\SS$ are uniformly bounded;
\end{enumerate}
then the minimal $G$-invariant set with respect to $\SS$ that contains the origin is \emph{bounded}. 
\end{theorem}

In addition, we have the following variants of Theorems \ref{theo:min_G} and \ref{theo:bound_G} on the existence of minimal \emph{convex} invariant sets.

\begin{definition}[Convex set operator induced by a collection of sets] \label{def:g_conv}
Fix $d \in \natnum$.
For any 
set 
$\S \subseteq \reals^d$, define:
\[
\mathfrak{G}_\S(Q):=\conv{\left(\mathfrak{g}_\S(Q)\right)} = \conv{\left( \bigcup_{c \in \S } \vsect{(\conv{\S}+Q)} -c\right)}.
\]
For a collection $\SS$, let
\[
\mathfrak{G}_\SS(Q):=\bigcup_{\S \in \SS} \mathfrak{G}_\S(Q).
\]
\end{definition}

\begin{proposition} \label{prop:inv_conv}
A \emph{convex} set $Q \subseteq \reals^d$ is invariant with respect to $\SS$ if and only if 
\[
Q = \mathfrak{G}_\SS(Q)
\]
\end{proposition}

\begin{theorem}[Minimal Convex Invariant Set] \label{theo:min_G_conv}
Let $Q \subseteq \reals^d$. The following statements hold:
\begin{enumerate}
    \item[(i)] The iterates
\begin{align*}
\mathfrak{G}^n_\SS (Q) &:= \mathfrak{G}_\SS (\mathfrak{G}^{n-1}_\SS (Q)) \quad \text{for } n\in \natnum, n \geq 1,\\ \mathfrak{G}^0_\SS (Q)& := \conv Q,
\end{align*}
are monotonic, in the sense that $\mathfrak{G}^n_\SS (Q) \subseteq \mathfrak{G}^{n'}_\SS (Q)$ for all $n \leq n'$, and the limit set 
\[
\mathfrak{G}^\infty_\SS (Q) := \lim_{n \rightarrow \infty} \mathfrak{G}^n_\SS (Q) = \bigcup_{n\geq0} \mathfrak{G}^n_\SS(Q)
\]
is the \emph{minimal convex invariant set} with respect to $\SS$ that contains $Q$.
\item[(ii)] Under the conditions of Theorem \ref{theo:bound_G}, $\mathfrak{G}^\infty_\SS (\{0\})$ is a bounded set.
\end{enumerate}
\end{theorem}

We note that the choice between the convex iteration of Theorem \ref{theo:min_G_conv} and the original iteration of Theorem \ref{theo:min_G} reflects the trade-off between a) performing convex hull at every iteration to keep the number of vertices describing the iterate at a minimum, and thereby lowering the cost of computing the rest of each iteration (intersection, union, and Minkowski-sum operations), and b) not computing convex hull in each step (hence saving this computational cost), at the expense of having a (possibly) non-convex iterate and a potential increase in the number of vertices needed to represent it, which in turn could lead to an increased computational cost of the iteration as a whole.








\subsubsection{Computational Method for Computing an Invariant Set}
The iteration of Theorem \ref{theo:min_G_conv} can 
be turned into a computational method (an algorithm that does not necessarily terminate)
by augmenting the iteration with the stopping rule that corresponds to the invariance property (\refprop{inv_conv}): ending the iteration when the vertex-representation of $\mathfrak{G}^{n+1}_\SS (Q)$ equals that of $\mathfrak{G}^{n}_\SS (Q)$.

We have implemented the method for the special case of point sets in C++ with the help of the CGAL library \cite{cgal}, the source code is available online \cite{boumangit}.
To prevent loss of precision during the iterations, and to be able to perform exact equality tests, we use exact rational arithmetic, instead of floating-point arithmetic.
%
%
Note that this choice restricts all vertices to have coordinates in $\mathbb{Q}$.

A problem with exact rational arithmetic is that 
a vertex coordinate might \emph{approach} a
mixed number (a sum of an integer and a proper fraction) whose fractional part has small numerator and denominator (by ``small'' we mean just a few digits),
like $1/2$, $5\tfrac13$, $300\tfrac56$, etc., but never reach it in finite time. 
To mitigate this problem, we apply, in every iteration, the following conditional rounding function to each coordinate of every vertex of $\mathfrak{G}^{n}_\SS (Q)$:
for a given $\epsilon \in \mathbb{R}$ and a finite set of proper fractions with small numerators and denominators $\mathcal{X}\subset\mathbb{Q}$, we define
\begin{align*}
\mathcal{R}:  \mathbb{Q} &\rightarrow \mathbb{Q} \\
q & \mapsto \begin{cases}
\lfloor q \rfloor + t & \text{if } |t- (q\!\mod 1) | \leq \epsilon, \\
q & \text{otherwise},
\end{cases}
\end{align*}
where  $t:= \argmin_{x \in \mathcal{X}}|t-(q\!\mod 1)|$.

Recall that it immediately follows from the stopping rule that if the method converges, it means that it has found an invariant set. Hence, we may in principle perturb the set in an arbitrary way after each iteration in an attempt to aid convergence. The ``rounding trick'' outlined above works well in practice. 

Note, however, that by perturbing the set 
$\mathfrak{G}^{n}_\SS (Q)$ (through $\mathcal{R}$) during the iteration, we cannot guarantee anymore that the method finds the \emph{minimal} invariant set. Nonetheless, if the method converges  and  
coordinate rounding occurs \emph{only} just before convergence, in other words, if 
the last ``rounding-free'' iterate is $\delta$-close (measured by a suitable metric for sets)   
to the invariant set found by the method, then that invariant set is 
a $\delta$-close approximation to the minimal invariant set, by the monotonicity the iterates.

An additional benefit of applying $\mathcal{R}$ 
is that the method is likely to  find an invariant set whose vertex-coordinates have small representation. 


\subsubsection{Numerical Examples in $\mathbb{R}^2$}
In a first example, we let $\mathbb{S} = \{ \mcal{S} \}$, with $\mcal{S} = \mcal{X}\times \mcal{Y}$, where $\mcal{X} = \{-1,1,3,5\}$ and $\mcal{Y}:=\{-1,1\}$, i.e., eight points on a rectangular grid. \reffig{reg} shows these vertices and the minimal invariant error set (in gray), which was found by our computational method (based on Theorem~\ref{theo:min_G_conv}) after one iteration.

In our second example, we consider  
$\mathbb{S}' = \{ \mcal{S}_1,  \mcal{S}_2,  \mcal{S}_3 \}$, with 
\begin{align*}
\mcal{S}_1 &:= \{ (-1, -1),(0, -1),(1, -1),(1, 0),(1, 1),(0, 1),(-1, 1), (-1, 0) \}, \\
\mcal{S}_2 &:= \mcal{S}_1 \setminus \{(0,-1)\}, \\ 
\mcal{S}_3 &:= \mcal{S}_2 \setminus \{(-1,-1)\}.
\end{align*}
In words: the set $\mcal{S}_1$ is a collection of points that are placed equidistantly on a rectangle; see \reffig{reggrid}, and the set $\mcal{S}_2$ and $\mcal{S}_3$ respectively are created by removing one resp. two points from $\mcal{S}_1$. In particular, it holds that $\mcal{S}_3 \subset \mcal{S}_2 \subset \mcal{S}_1$. This could correspond to a setting  in practice where the local controller can implement points from $\mcal{S}_1$ most of the time, but once in a while one particular setpoint, (and sometimes even an additional particular setpoint) becomes temporarily infeasible. 

The minimal invariant error set, shown in \reffig{invset}, has $7$ vertices and was found after $177$ iterations, with $\epsilon = 10^{-8}$ as rounding parameter.
From this figure, we see that the minimal invariant error set corresponding to $\mathbb{S}'$ (the ``joint'' error set) is significantly larger than the minimal invariant error sets corresponding to singletons $\mathbb{S}''=\{\mcal{S}_i\}$ for all $i\in \{1,2,3\}$. Hence, if having a small invariant error-set is of central importance in an application, then one could, in this particular example, decide to only use the setpoints in $\mcal{S}_3$ at the cost of providing a smaller feasible set (on average) to the central controller.

\begin{figure}
  \centering
  \includegraphics{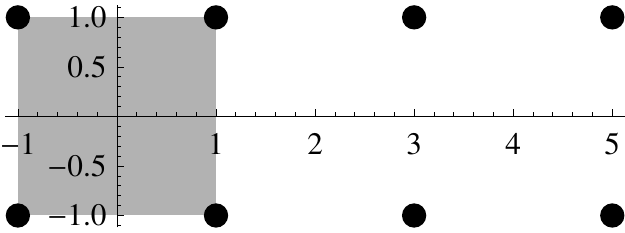}
\caption{A single-point-set example. The minimal invariant error set is shown in light gray.} \label{fig:reg}
\end{figure}

\begin{figure}
  \centering
  \includegraphics{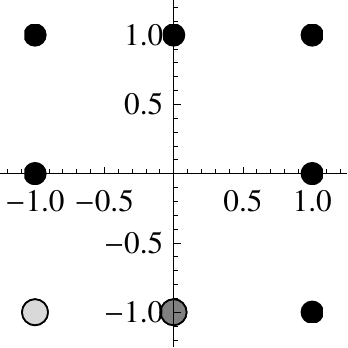}
\caption{The set of setpoints $\mcal{S}_1$ (all points). The gray and lightgray points indicate the setpoints that are removed from $\mcal{S}_1$ to construct $\mcal{S}_2$ and $\mcal{S}_3$.} \label{fig:reggrid}
\end{figure}

\begin{figure}
  \centering
  \includegraphics{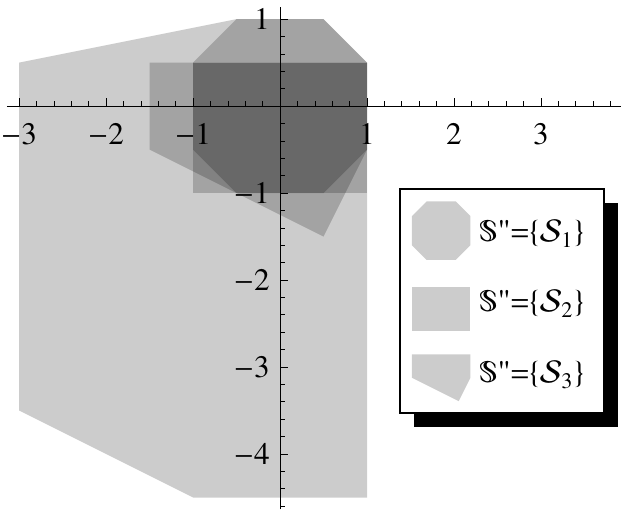}
\caption{The minimal invariant error set for  $\mathbb{S}'\!=\!\{\mcal{S}_1,\mcal{S}_2,\mcal{S}_3\}$; a convex polygon with corner-vertices 
$\{
 (-3 , -\frac{7}{2} ),
 (-1 , -\frac{9}{2} ),
 (1 , -\frac{9}{2} ),
 (1 , \frac{1}{2} ),
 (\frac{1}{2} , 1 ),
 (-\frac{1}{2} , 1 ),
 (-3 , \frac{1}{2}
\}$. In the same figure, the invariant error sets are drawn corresponding to the cases  where $\mathbb{S}''=\{\mcal{S}_i\}$, for all $i \in \{1,2,3\}$. Those error sets are much smaller than the ``joint'' invariant error set.} \label{fig:invset}
\end{figure}

\subsection{Persistent Prediction} \label{sec:persist}
We now analyze the case of using a persistent predictor in the local controller (step 3(c) of Algorithm \ref{alg:inter}), where the advertised predicted feasible set is given by $\A_{n+1} = \conv \S_n$. We then have the following error dynamics
\[
e_{n+1} = e_n + x_n - \vor[\S_n]{e_n + x_n}
\]
where $x_n \in \A_n = \conv{\S_{n-1}}$ and $\S_n, \S_{n-1} \in \SS$. 
Observe that this dynamics involves a pair of sets rather then a single set, and hence 
the previous results cannot be applied directly. Fortunately, we can reformulate this dynamics in terms of the modified request $z_n = e_n + x_n$ similarly to the approach in \cite{adler2005,nowicki2004}. For this new state variable, we have that
\begin{equation} \label{eqn:x_dyn}
z_{n+1} = z_n + x_{n+1} - \vor[\S_n]{z_n},
\end{equation}
where $x_{n+1} \in \A_{n+1} = \conv{\S_{n}}$. Thus, the following operator can be defined.

\begin{definition}
For any finite non-empty  set $\S \subset \reals^d$ and any $x \in \conv{\S}$ we define the map
\begin{align*}
F_{\S,x }:  \reals^d &\rightarrow \reals^d  \\
z & \mapsto z + x - \vor[\S]{z}
\end{align*}
\end{definition}

The dynamics \eqref{eqn:x_dyn} can be then expressed as
\begin{equation} \label{eqn:x_dyn2}
z_{n+1} = F_{\S_n,x_{n+1}}(z_n)
\end{equation}
and $F$-invariance with respect to $\SS$ is given by Definition \ref{def:invariance} by replacing $G$ with $F$.

The following proposition makes the connection between $F$-invariant sets and boundedness of the accumulated error.
\begin{proposition} \label{prop:inv_bound_domain}
Let $\SS$ be a collection of subsets of $\reals^d$. Consider the dynamics \eqref{eqn:x_dyn2}, where $\S_n \in \SS$ for all $n$. Let $D$ be an $F$-invariant set with respect to the collection $\SS$, and assume that $z_0 \in D$.
\begin{enumerate}
    \item[(i)] Then it holds that 
    $$\|e_n\| \leq \max_{\S \in \SS} \max_{v \in D - \conv \S} \|v\|$$ 
    for all $n \geq 1$.
    \item[(ii)] If in addition $\S \subseteq D$ for all $\S \in \SS$, we have that $\|e_n\| \leq \diam D$ for all $n \geq 1$.
\end{enumerate}
\end{proposition}

\begin{proof}
Part (i) of the proposition follows trivially by the invariance of $D$, and the fact that $e_n = z_n - x_n$ for $z_n \in D$ and $x_n \in \conv \S$ for some $\S \in \SS$.

For part (ii), observe that
\[
e_{n+1} = x_n + e_n - \vor[\S_n]{x_n + e_n} = z_n - \vor[\S_n]{z_n}
\]
where both $z_n \in D$ and $\vor[\S_n]{z_n} \in \S_n \subseteq D$. Hence, $\|e_{n+1}\| \leq \diam D$.
\end{proof}

We next state our main result that provides conditions for the existence of bounded minimal $F$-invariant sets, and a method to compute these sets. Hence we obtain tight bounds for the accumulated error in the case of persistent prediction. 
To that end, we first define the following set-operators similarly to \cite{nowicki2004}.

\begin{definition}[Set operators] \label{def:p}
Fix $d \in \natnum$.
For any set $\S \subseteq \reals^d$, define:
\[
\mathfrak{p}_\S(D):=\P + \bigcup_{c \in \S } \vsect{D} -c, \quad \mathfrak{P}_\S(D):=\conv{\left(\mathfrak{p}_\S(D)\right)}.
\]
Also, for a collection \SS, let
\[
\mathfrak{p}_\SS (D):=\bigcup_{\S \in \SS} \mathfrak{p}_\S(D), \quad \mathfrak{P}_\SS (D):=\conv{\left(\mathfrak{p}_\SS (D)\right)}.
\]
Also, define the iterates of the above operators, $\mathfrak{p}^n_\S(D), \mathfrak{P}^n_\S(D), \mathfrak{p}^n_\SS(D)$ and $\mathfrak{P}^n_\SS(D)$, as before.
\end{definition}

\begin{proposition} \label{prop:inv_mult_p}
\begin{enumerate}
\item[]
    \item[(i)] For any set $D \subseteq \reals^d$, $D \subseteq \mathfrak{p}_\SS(D)$, and $D$ is $F$-invariant with respect to $\SS$ if and only if 
\[
\mathfrak{p}_\SS(D) \subseteq D.
\]
\item[(ii)] For any \emph{convex} set $D \subseteq \reals^d$, $D \subseteq \mathfrak{P}_\SS(D)$,  and convex $D$ is $F$-invariant with respect to $\SS$ if and only if 
\[
\mathfrak{P}_\SS(D) \subseteq D.
\] 
\end{enumerate}
\end{proposition}

\begin{theorem} \label{thm:main_uncertain}
Let \SS be a collection of sets $\S \subset \reals^d$, and let $D$ be any given set. 
The following statements hold:
\begin{enumerate}
    \item[(i)] The iterates $\mathfrak{p}^n_\SS(D), \mathfrak{P}^n_\SS(D)$ are monotonically non-decreasing, and the sets
\[
\mathfrak{p}^\infty_\SS(D) := \lim_{n \rightarrow \infty} \mathfrak{p}^n_\SS(D) \quad \text{and} \quad \mathfrak{P}^\infty_\SS(D) := \lim_{n \rightarrow \infty} \mathfrak{P}^n_\SS(D)
\]
are the minimal $F$-invariant set and the minimal \emph{convex} $F$-invariant set, respectively, containing the set $D$. 
\item[(ii)] Under the conditions of Theorem \ref{theo:bound_G}, both $\mathfrak{p}^\infty_\SS(\{0\})$ and $\mathfrak{P}^\infty_\SS(\{0\})$ are bounded sets.
\end{enumerate}
\end{theorem}

\section{Application Examples and Numerical Illustration}

In this section, we give concrete examples of the design of local controllers for some important types of resources in the context of real-time control of electrical grid. 
We also perform numerical simulation of the \cl system \cite{commelec} that includes one central controller (grid agent) and several local controllers (resource agents), and show how boundedness of the accumulated error helps to achieve better overall performance of the closed-loop \cl system.

\subsection{Discrete Resource in One Dimension}
In this section, 
we show how our general results apply to local controllers that 
can only implement setpoints from a \emph{discrete} and \emph{one-dimensional} set.
As an application, we consider a heating system consisting of a finite number of heaters that each can either be switched on or off (see \refsec{building} below). In particular,  this system can only produces/consumes \emph{real power}, and thus its set of feasible setpoints is one-dimensional.
We focus here on \emph{deterministic} systems, hence we are in the case of \emph{perfect prediction}.

Consider a \emph{finite} collection $\SS = \{\S^k\}_{k \in [K]}$ of \emph{finite} non-empty subsets $\S^k \subseteq \reals$, where $K \in \natnum$. 
For any $\S \in \SS$, whose elements we label as $s_1 < s_2 < \ldots < s_{|\S|}$, we let
  \[
    \Delta_\S \defas
      \begin{cases}
        0 & \text{if $|\S|=1$}\\
         \max_{i \in [|\S|-1]} s_{i+1} - s_i & \text{if $|\S|>1$}.
  \end{cases}
  \]
denote the \term{maximum stepsize of \S}. We also let
\begin{equation} \label{eqn:collstep}
\Delta_\SS \defas \max_{\S \in \SS} \Delta_{\S}
\end{equation}
denote the \term{maximum step size of the collection $\SS$}.

\begin{theorem} \label{thm:discrete}
Let $\SS$ be a finite collection of finite non-empty subsets of $\reals$. Then the set $[-\Delta_\SS/2, \Delta_\SS/2]$ is the minimal $G$-invariant set with respect to $\SS$ (as per Definition \ref{def:invariance}) that contains the origin. Hence, if $e_0 \in [-\Delta_\SS/2, \Delta_\SS/2]$, the accumulated error under the greedy algorithm \eqref{eqn:err_diff} is bounded by $|e_n| \leq \Delta_\SS/2$ for all $n \geq 1$.
\end{theorem}

\begin{proof}
We use the iteration of Theorem \ref{theo:min_G} to show that $\mathfrak{g}^\infty_\SS (\{0\}) = [-\Delta_\SS/2, \Delta_\SS/2]$. For any $\S \in \SS$, it is easy to see that
\[
\mathfrak{g}_\S(\{0\}) = \bigcup_{c \in \S } \vsect{\P} -c = [-\Delta_\S/2, \Delta_\S/2].
\]
Indeed, consider the two points in $\SS$ that attain the maximum step size, namely $s_i$ and $s_{i+1}$ such that $s_{i + 1} - s_i = \Delta_\S$. Then
\[
\left(\P_{\{s_i\}} - s_i\right) \bigcup \left(\P_{\{s_{i+1}\}} - s_{i+1}\right) = [-\Delta_\S/2, \Delta_\S/2]
\]
and the rest of the terms $\vsect{\P} -c$  are contained in $[-\Delta_\S/2, \Delta_\S/2]$. Therefore, the first iteration yields
\[
\mathfrak{g}_\SS(\{0\}) =\bigcup_{\S \in \SS } [-\Delta_\S/2, \Delta_\S/2] = [-\Delta_\SS/2, \Delta_\SS/2] := Q.
\]

For the second iteration, for any $\S \in \SS$, consider
\begin{equation} \label{eqn:sec_iter}
\mathfrak{g}_\S^2(\{0\}) = \mathfrak{g}_\S(Q) = \bigcup_{c \in \S } \vsect{(\P + Q)} -c.
\end{equation}
Denote $\P = [s_1, s_{|\S|}]$. Observe that for any $s_i \in \S$, $i \neq 1, |\S|$, 
\begin{align*}
(\P + Q)_{\{s_i\}} &= V_{\S}(s_i) \\
&= [s_i - (s_i - s_{i-1})/2,  s_i + (s_{i+1} - s_i)/2] \\
&\subseteq [s_i - \Delta_\SS/2, s_i + \Delta_\SS/2]
\end{align*}
by the definition of the Voronoi cell of $s_i$ and of the maximal step size of the collection $\SS$. For $c = s_1$,
\begin{align*}
(\P + Q)_{\{s_1\}}  &= [s_1 - \Delta_\SS/2, s_1 +  (s_2 - s_1)/2]\\
 & \subseteq [s_1 - \Delta_\SS/2, s_1 + \Delta_\SS/2],
\end{align*}
and similarly
\begin{align*}
(\P + Q)_{\{s_{|\S|}\}}  &= [s_{|\S|} -  (s_{|\S|} - s_{|\S| - 1})/2, s_{|\S|} + \Delta_\SS/2] \\
&\subseteq [s_{|\S|} - \Delta_\SS/2, s_{|\S|} + \Delta_\SS/2].
\end{align*}
Hence, for all $c \in \S$, $\vsect{(\P + Q)} - c \subseteq [-\Delta_\SS/2, \Delta_\SS/2]$. Substituting in \eqref{eqn:sec_iter} yields
\[
\mathfrak{g}_\S^2(\{0\}) = \mathfrak{g}_\S(Q) \subseteq Q
\]
and consequently  $\mathfrak{g}_\S(Q) =  Q$ by \refprop{inv_mult}. Therefore,
\[
\mathfrak{g}_\SS^2(\{0\}) = \mathfrak{g}_\SS(Q) = \bigcup_{\S \in \SS } \mathfrak{g}_\S(Q) = Q.
\]
Thus the iteration has converged and, by Theorem \ref{theo:min_G}, $\mathfrak{g}_\SS^{\infty}(\{0\}) = Q := [-\Delta_\SS/2, \Delta_\SS/2]$ is the minimal invariant set.
\end{proof}

\subsection{Example: Resource Agent for Heating a Building} \label{sec:building}

\newcommand{\Ttgt}{\ensuremath{T_\textnormal{target}}\xspace}
\newcommand{\Tmin}{\ensuremath{T_\textnormal{min}}\xspace}
\newcommand{\Tmax}{\ensuremath{T_\textnormal{max}}\xspace}
\newcommand{\Pheat}{\ensuremath{P_\textnormal{heat}}\xspace}

In this section, we present a concrete local-controller example: we will design a local controller for managing the temperature in a building with several rooms.
The reason for showing this example is twofold.
First, we wish to give a concrete example of a local controller that
controls a load that can only
implement power setpoints from a discrete set. Second, the local-controller design shows a concrete usage example of the \cl framework \cite{commelec},
and might serve as a basis for an actual implementation.

In this section and next section, we will use the terminology of the \cl framework and denote the local controller as ``resource agent'', while the advertised  feasible set of power setpoints $\A_n$ as ``\pqprof'', where $P$ and $Q$ denote active and reactive power, respectively.

The heating system's objective is 
to keep the rooms' temperatures within a certain range.
For rooms whose temperature lies in that range, there is some freedom in the choice of the control actions related to  those rooms.
The resource agent's job is to monitor the building and spot such degrees of freedom, and expose them to the grid agent, which can then exploit those for performing \emph{Demand Response}.

Our example is inspired by \cite{costanzo}, which also considers the problem of controlling the temperature in the rooms of a building using multiple heaters.
We address two issues that were not addressed in \cite{costanzo}:
\begin{enumerate}
\item We show 
  that by rounding requested setpoints into implementable setpoints using the error diffusion algorithm
  we obtain a resource agent with bounded accumulated-error.
\item We prevent the heaters from switching on and off with the same frequency as \cl's control frequency, which is crucial in an actual implementation. 
\end{enumerate}

\subsubsection{Simple Case: a Single Heater}
For simplicity, we first analyze a scenario with only one heater.
The main aspects of our proposed design (as mentioned above) are in fact independent of the number of heaters, and we think that those aspects are more easily understood in this simple case.
We will generalize our example to an arbitrary number of heaters in \refsec{moreheaters}.

\paragraph{Model and Intended Behavior}
We model the heater as a purely resistive load (it does not consume reactive power)
that can be either active (``on'') or inactive (``off''). It consumes $P_\text{heat}>0$ Watts while being active, and zero Watts while being inactive.

From the perspective of the resource agent, 
the heater has a \emph{state} that consists of two binary variables:
$s_n \in \{0,1\}$, which corresponds to whether the heater is on ($s_n=1$) or off ($s_n=0$), and
$\ell_n \in \{0,1\}$ indicates whether the heater is ``locked'', in which case we cannot switch on or switch off the heater.
Formally, if $\ell_n = 1$ then $s_{n+1} = s_{n}$ necessarily holds.
Hence, $\ell_n$ exposes a physical constraint of the heater, namely that it cannot (or should not) be switched on and off with arbitrarily high frequency.
In the typical case where the minimum switching period of the heater is (much) larger than the \cl's control period ($\approx 100$ ms), the heater will ``lock'' immediately after a switch, i.e., assuming $\ell_n=0$, setting $s_{n+1}$ such that $s_{n+1}\neq s_n$ will induce
$\ell_{n'} = 1$ for every $n' \in [n+1, n + N]$, after which $\ell_{n + N+1} = 0$. Here, $N\in \natnum$ represents the minimum number of timesteps  for which the heater cannot change its state from on to off or \emph{vice versa}.

Suppose that the heater is placed in a room, and that the temperature of this room is a scalar quantity. (We do not aim here to model heat convection through the room or anything like that.)
The temperature in the room, denoted as $T_n$, should remain within predefined ``comfort'' bounds,
\begin{align*}
T_n \in [ \Tmin,\Tmax] \quad n=1,2,\ldots
\end{align*}
where $\Tmin,\Tmax \in \reals$.
If $T_n$ is outside this interval (and only if $\ell_n=0$), then the resource agent should take the trivial action, i.e.,
ensure that the heater is active if $T_n < \Tmin$, and
 inactive if $T_n > \Tmax$.
The more interesting case is 
if $T_n$ lies in $[\Tmin,\Tmax]$ (again, provided that $\ell_n=0$), as this gives rise to some \emph{flexibility} in the heating system:
the degree of freedom here is whether to switch the heater on or off, which obviously directly corresponds to the
total power consumed by the heating system.
The goal is to
delegate this choice to 
the grid agent, 
which we can accomplish by defining an appropriate \cl advertisement.

\paragraph{Defining the Advertisement and the Rounding Behavior} 
\label{sec:radef}
Let the discrete set of implementable real-power setpoints 
at time $n$ be defined as
\[
  \S_n\defas  \begin{cases}
  \{ 0 \} & \text{if }(\ell_n=0 \land T_n >\Tmax) \lor \\
          & \phantom{\text{if }}( \ell_n = 1 \land s_n=0),\\
    \{-P_\textnormal{heat},0\} &\text{if }\phantom{(}\ell_n=0 \land \Tmin\leq T_n \leq\Tmax, \\
 \{ -P_\textnormal{heat} \} & \text{if } (\ell_n=0 \land T_n <\Tmin) \lor \\
                           & \phantom{\text{if }}( \ell_n = 1 \land s_n=1),
  \end{cases}
\]
where $\land$ and $\lor$ stand for ``and'' and ``or'', respectively.
Note that $\S_n$ only contains non-positive numbers, 
by the convention in \cl that \emph{consuming} real power corresponds to \emph{negative} values for $P$.
We define the \pqprof as a perfect prediction of $\conv \S_n$, namely $\A_n = \conv \S_n$, and assume that the resource agent uses the error diffusion algorithm \eqref{eqn:err_diff} to implement setpoints $y_n$.
We then have the following immediate corollary of \refthm{discrete}.

\begin{corollary}
  The accumulated error of the single-heater resource agent as defined in this section is bounded by 
  $\tfrac12 \Pheat$.
  \label{cor:sing}
\end{corollary}

\subsubsection{General Case: an Arbitrary Number of Heaters}
\label{sec:moreheaters}
Here, we extend the single-heater case to a setting with $r$ heaters, for $r\in \natnum, r\geq 1$ arbitrary. 
As we will see, also this multi-heater case can be analyzed using \refthm{discrete}.

Like in the single-heater case, we assume that each heater is purely resistive. We furthermore assume that heater $i$ consumes $P_i^\text{heat}$ Watts of power when active (and zero power when inactive), for every $i\in [r]$.
Also similarly to the single-heater case, we assume that each heater is placed in a separate room, whose (scalar) temperature is denoted as $T_n^{(i)}$.
Not surprisingly, our objective
shall now be to keep the temperature in each room within the predefined comfort bounds,
i.e.,
\begin{align*}
  T_n^{(i)} \in [ \Tmin,\Tmax] \quad \forall n \in \posint, \forall i \in [r].
\end{align*}

In the one-heater case, the only degree of freedom 
is the choice to switch that heater on or off.
In case of multiple heaters, there is potentially some freedom in
choosing which subset of the heaters to activate,
and note that there will typically\footnote{Provided that not too many heaters are locked.} be an exponential number of those subsets (exponential in the number of heaters).
Each subset corresponds to a certain total power consumption, i.e., a power setpoint. 
As in the single-heater case,
the \pqprof will be defined as the convex hull of the collection of these setpoints.
When the grid agent requests some setpoint from the \pqprof, the resource agent has to select an appropriate subset whose corresponding setpoint is closest (in the Euclidean sense) to the requested setpoint.
Note that there can be several subsets of heaters that correspond to the same setpoint.
A simple method to resolve this ambiguity would be, for example, to choose the subset consisting of the coldest rooms, however, as this topic is beyond the scope of this work, we leave the choice of such a selection method to the resource-agent designer.

When going from the single-heater setting to a multiple-heaters scenario,
we merely need to re-define $\S_n$, which we will name  
$\widetilde{\S}_n$ here to avoid confusion with the single-heater case. 
The definition of the \pqprof, 
and rule for computing $y_n$ given in \refsec{radef}
also apply to the multi-heater case, provided that
all occurrences of $\S_n$ in those definitions are replaced by $\widetilde{\S}_n$.

For every $i\in [r]$, let $s_n^{(i)}$ and $\ell_n^{(i)}$ represent the state variables $s_n$ and $\ell_n$ (as defined in the single-heater case) for the $i$-th heater. Let $\mcal{L}_n \defas  \{ i \in [r]: \ell_n^{(i)} = 1 \}$ denote the set of rooms whose heater is locked at timestep $n$. Furthermore, let $\mcal{C}_n\defas  \{ i \in [r]: T_n^{(i)} < \Tmin \}$
and $\mcal{W}_n\defas  \{ i \in [r]: \Tmin \leq T^{(i)}_n \leq \Tmax \} $. 
Informally speaking, $\mcal{C}_n$ contains the rooms that are ``too cold'', and $\mcal{W}_n$ the rooms whose temperatures are within the comfort bounds.

If $A \subseteq[r]$, we write $\overline{A}$ for the complement
with respect to $[r]$, i.e. $\overline{A}\defas [r]\setminus A$.

Let
\begin{equation} \label{eqn:B_n}
  \widetilde{\S}_n : = \big\{ a_n - \sum_{i \in \mcal{I} }  P_i^\textnormal{heat}  : \mcal{I} \subseteq \overline{\mcal{L}_n}\cap\mcal{W}_n \big\}
\end{equation}
represent the set of implementable (active) power setpoints, with
\[
a_n \defas  -\sum_{i \in \mcal{L}_n} s_n^{(i)} P_i^\textnormal{heat} - \sum_{j \in \overline{\mcal{L}_n} \cap \mcal{C}_n } P_j^\textnormal{heat}.
\]

\newcommand{\tS}{{\widetilde{\mathbb{S}}}}
As in the one-heater example, we use \refthm{discrete} to bound the accumulated error of the resource agent.
To this end, let $\tS$ denote the collection of all possible sets
$\widetilde{\S}_n$ \eqref{eqn:B_n}. It is easy to see that this is a finite collection. Further, the maximum stepsize of this collection \eqref{eqn:collstep} is given by $\Delta_\tS = \max_{i \in [r]} P_i^\textnormal{heat}$. 
This gives us the following corollary.

\begin{corollary}
  The accumulated error of the multiple-heaters resource agent as defined in this section is bounded by  $(\tfrac12 \max_{i \in [r]} P_i^\textnormal{heat})$. 
\end{corollary}

\subsection{Example: Resource Agent for a Photovoltaic (PV) System}
\label{sec:pvexample}
Here, we explain how we can apply \refthm{main_uncertain} to
devise a local controller (or a \emph{resource agent} in the terminology of \cite{commelec})
for a PV system with bounded accumulated-error.

Let $S_\text{rated}$ and $\phi_{\max}$ denote
the rated power of the converter and angle corresponding to the minimum power factor, respectively. We suppose that these quantities are given (they correspond to physical properties of the PV system), and that $S_\text{rated}\geq 0$ and $0 \leq \phi_{\max} < \pi$. 
Note that for any power setpoint $(P,Q)$, the rated power imposes the constraint $P^2+Q^2 \leq  S^2_\text{rated}$; the angle $\phi_{\max}$ imposes that $\arctan (Q/P) \leq \phi_{\max}$.

Let us now choose $P_{\max},  \varphi \in \reals$ such that $0 \leq P_{\max} \leq S_\text{rated}$, $ 0 \leq \varphi \leq \phi_{\max}$, and $\frac{P_{\max}}{\cos \varphi} = S_\text{rated}$, and let
\[
  \mcal{T}(x) :=  \{ (P,Q)\in \reals^2 : 0 \leq P \leq x, \frac{|Q|}{P}\leq \tan \varphi \}
\]
be a triangle-shaped set in the $PQ$ plane; see \reffig{Tset} for an illustration.
Note that for any combination of $P_{\max}$ and $\varphi$, the triangle $\mcal{T}(x)$ for any $x\in [0,P_{\max}]$ is fully contained in the disk that corresponds to the rated-power constraint, and, moreover, the two upper corner points of $\mcal{T}(P_{\max})$  lie on the boundary of that disk.

Let $p^{\max}_n$ be the maximum real power available at timestep $n\in \natnum$ (typically determined by the solar irradiance). 
Using $p^{\max}_n$, we define the set of implementable points
at timestep $n$ as
\[
\S_n:= \mcal{T}(\min(p_n^{\max},P_{\max})).
\]

\begin{figure}
  \centering
  \includegraphics[scale=.8]{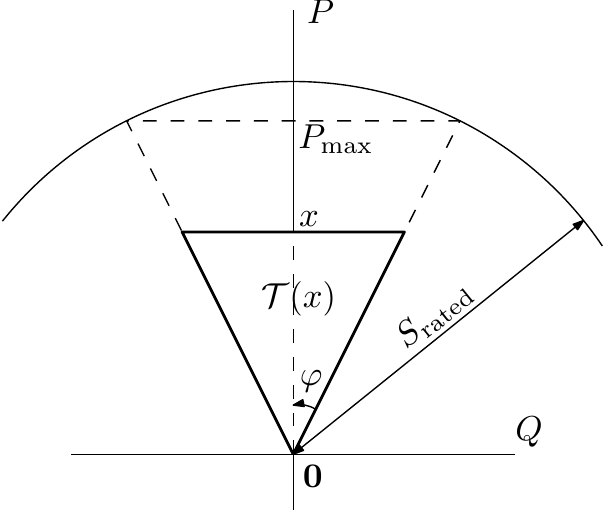}
  \caption{The parameterized collection of sets $\{ \mcal{T} (x): x\in [0,P_{\max}]\}$, and its relation to the rated-power constraint of the PV converter.}
  \label{fig:Tset}
\end{figure}

\begin{theorem} \label{thm:pv}
Let $P_{\max}$, $\varphi$, and $\S_n$ (for every $n\in \natnum$) be defined as above for a given PV system. 
    Consider a resource agent for this PV system
    that: (i) uses a persistent predictor to advertise $\A_{n+1} := \S_{n}$, and (ii) implements setpoints according to the greedy (error diffusion) algorithm. Then, $D = \mcal{T}(P_{\max})$ is the minimal $F$-invariant set with respect to $\SS = \{\mcal{T}(x) \}_{0 \leq x \leq P_{\max}}$. Therefore, if $e_0 \in D$, the accumulated error for this resource agent is bounded by
    \[
    \|e_n\| \leq \diam D = \max\l\{ \frac{P_{\max}}{\cos\varphi}, 2\, P_{\max}\tan \varphi \r\}
    \]
    for all $n \in \natnum$.
\end{theorem}

The proof of \refthm{pv} relies on the following general result.

\begin{lemma} \label{lem:union_min}
Let $\SS$ be a collection of non-empty subsets of $\reals^d$. Assume that $\bigcap_{\S \in \SS} \S \neq \emptyset$. Let 
\[
D := \bigcup_{\S \in \SS} \conv S.
\]
If $D$ is $F$-invariant with respect to $\SS$ then, for any $s_0 \in \bigcap_{\S \in \SS} \S$, $D$ is the \emph{minimal} $F$-invariant set with respect to $\SS$ that contains $s_0$. 
\end{lemma}

\begin{proof}
Let $s_0 \in \bigcap_{\S \in \SS} \S$. We use the iteration of \refthm{main_uncertain} to prove that $D$ is the minimal $F$-invariant set that contains $s_0$. For the first iteration, for every $\S \in \SS$, we have by \refdef{p} that
\begin{eqnarray*}
\mathfrak{p}_\S(\{s_0\}) &:=& \P + \bigcup_{c \in \S } \vsect{\{s_0\}} -c \\
&=& \P + s_0 - s_0 = \P,
\end{eqnarray*}
where the second equality follows by the fact that $s_0 \in \S$. Therefore, 
\[
\mathfrak{p}_\SS(\{s_0\}) := \bigcup_{\S \in \SS} \mathfrak{p}_\S(\{s_0\}) = \bigcup_{\S \in \SS} \P := D.
\]
Now for the second iteration,
\[
\mathfrak{p}_\SS^2(\{s_0\}) = \mathfrak{p}_\SS(D) = D
\]
by the invariance of $D$. Thus the iteration has converged, and by \refthm{main_uncertain}, $D$ is the minimal $F$-invariant set with respect to $\SS$ that contains $s_0$.
\end{proof}

\begin{proof}[Proof of \refthm{pv}]
Observe that $\bigcap_{\S \in \SS} \S = \{0\}$ and 
\[
\bigcup_{\S \in \SS} \conv S = \bigcup_{\S \in \SS}  S = \bigcup_{0 \leq x \leq P_{\max}} \mcal{T}(x) = \mcal{T}(P_{\max}).
\]
Hence, it remains to show that $\mcal{T}(P_{\max})$ is $F$-invariant with respect to $\SS$. The minimality property then follows immediately from \reflem{union_min}. 
By \refprop{inv_mult_p}, it is enough to show that
\[
\mathfrak{p}_\SS(\mcal{T}(P_{\max})) \subseteq \mcal{T}(P_{\max}).
\]
In other words, we want to show that for any $\S \in \SS$,
\[
\S + \bigcup_{c \in \S} \vsect{\mcal{T}(P_{\max})} - c \subseteq \mcal{T}(P_{\max}).
\]
%
%
First, note that it is straightforward to characterize the different types of Voronoi cells of \S (see also \reffig{triangleconstr2}): 
for interior points, the Voronoi cell is the point itself;
for non-corner points on the boundary, the Voronoi cell is the outward-pointing ray that emanates from that point and is normal to the facet;
for corner points on the boundary, the Voronoi cell is a cone, namely the union of all rays that emanate from that corner point, whose directions vary (continuously) between the normals of the adjacent facts.

Now, it is not hard to see that for any $x\in[0,P_{\max}]$, we can construct $
\mcal{G}(x) := \bigcup_{c \in \mcal{T}(x)} \vsect{\mcal{T}(P_{\max})} - c
$  as shown in \reffig{triangleconstr2}. From this construction it then immediately follows that 
$\mcal{T}(x) + \mcal{G}(x) = \mcal{T}(P_{\max})$, which  
proves $F$-invariance of $\mcal{T}(P_{\max})$ with respect to $\mathbb{S}$. 
Hence, from \refprop{inv_bound_domain} we then have that 
    $\|e_n\| \leq \diam D$
    for all $n \in \natnum$.
Because $\mcal{T}(P_{\max})$ is an isosceles triangle, its diameter is either $P_{\max}/\cos \varphi$ (the length of one of its \emph{legs}) for $\varphi \leq\frac{\pi}{6}$ or $ 2\,P_{\max}\tan \varphi$ otherwise (the length of its \emph{base}).
\end{proof}

\begin{figure*}
  \centering
  \includegraphics[scale=.8]{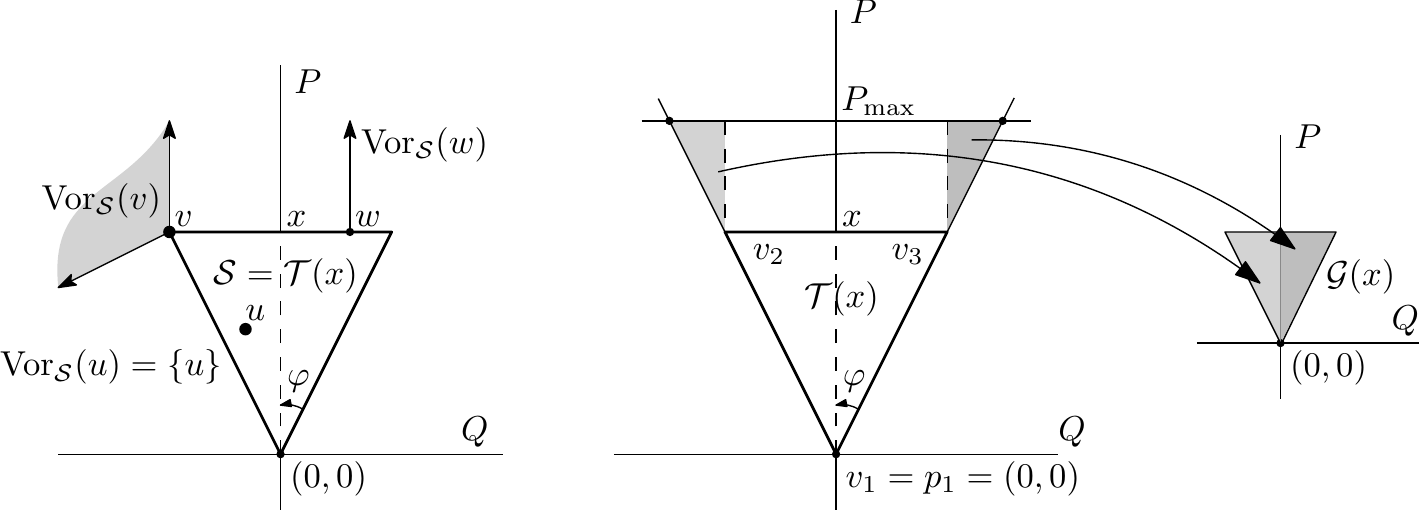}
  \caption{\emph{Left} -- The three Voronoi cell-types for the set \S, 1) a singleton for any interior point ($u$), 2) a cone for any corner point ($v$), and 3) a ray (normal to the corresponding facet) for any non-corner point on the boundary ($w$). \emph{Right} -- Construction of the set $\mcal{G}(x)$.}
  \label{fig:triangleconstr2}
\end{figure*}

\subsection{Simulation} \label{sec:num}

As in \cite{commelec2}, we take a case study that makes reference to the low voltage microgrid benchmark defined by the CIGR\'{E} Task Force C6.04.02 \cite{LVBenchmark}.
For the full description of the case study and the corresponding agents design, the reader is referred to \cite{commelec2}.

There are two modifications compared to the original case study: (i) The PV agents are updated with the algorithm described in Section \ref{sec:pvexample}, and (ii) instead of using an uncontrollable load in the case study, we use a resistive heaters system, and the corresponding agent is implemented according to the methods described in Section \ref{sec:building}.

We simulate a rather extreme scenario involving a highly variable solar irradiance profile. That is, we let the irradiance vary according to a square wave with a period of  $300$ ms. This will cause the PV agent's \pqprof to be highly variable. 

We let the cost function of the PV agent be the same as in \cite{commelec2}; this cost function  encourages to maximize active-power output. The cost function of the heater is set to a quadratic function, whose minimum lies at half the heater power, namely at $-7.5$ kW. With respect to the locking behavior of the heater, we let it lock for one second after a switch.

The results are shown in Figures \ref{fig:errorDiff_heater}--\ref{fig:errorDiffPV}. For comparison, we run the same scenario with resource agents
for which the accumulated error might grow unboundedly. I.e., those RAs do not apply the error-diffusion technique described in this paper, instead, they just project the request to the closest implementable setpoint, like in \cite{commelec2}.

\reffig{errorDiff_heater} and \ref{fig:convergence} illustrate how the use of the error-diffusion algorithm in the heater agent affects the convergence properties of the closed-loop system. Recall that as in \cite{commelec2}, the central controller (i.e., the grid agent) is using a gradient-descent algorithm to compute power setpoints. It can be seen from \reffig{errorDiff_heater} that without error-diffusion, the grid agent ``gets stuck'' on a setpoint that is not close to the optimal value, and the implemented setpoint is also suboptimal. On the other hand, with error-diffusion, both requested and implemented setpoints are optimal on average, as can also be seen in \reffig{convergence}. Similar behavior can be observed in \reffig{utilization} and \ref{fig:errorDiffPV} for the PV agent. Moreover, \reffig{utilization} shows how the error-diffusion algorithm helps to improve utilization of renewables. In particular, there is less curtailment of the PV, hence more energy is produced from this renewable energy source. Finally, the simulation results corroborate the theoretical analysis by showing delivery of the requested power on average, i.e., energy, which can be valuable for an application like a virtual power plant (\reffig{accerror_heater}, \ref{fig:convergence}, \ref{fig:accerror_PV} and \ref{fig:errorDiffPV}).

\begin{figure}
\centering
\includegraphics[width=.8\columnwidth]{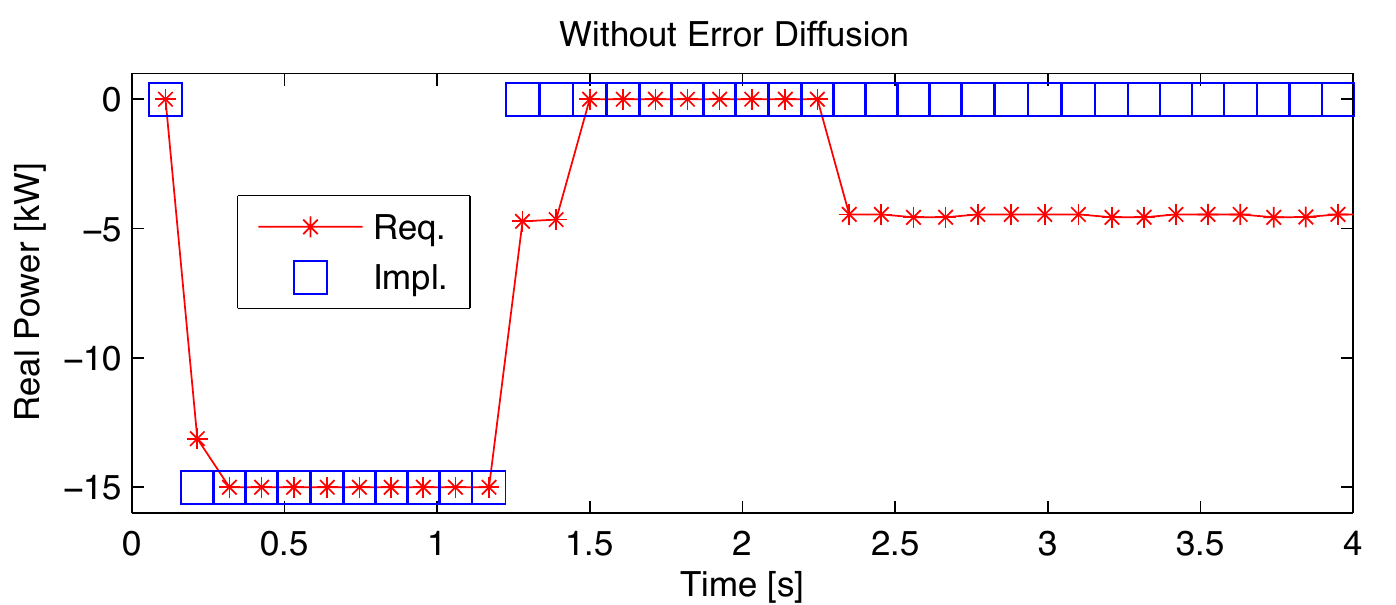}\\
\includegraphics[width=.8\columnwidth]{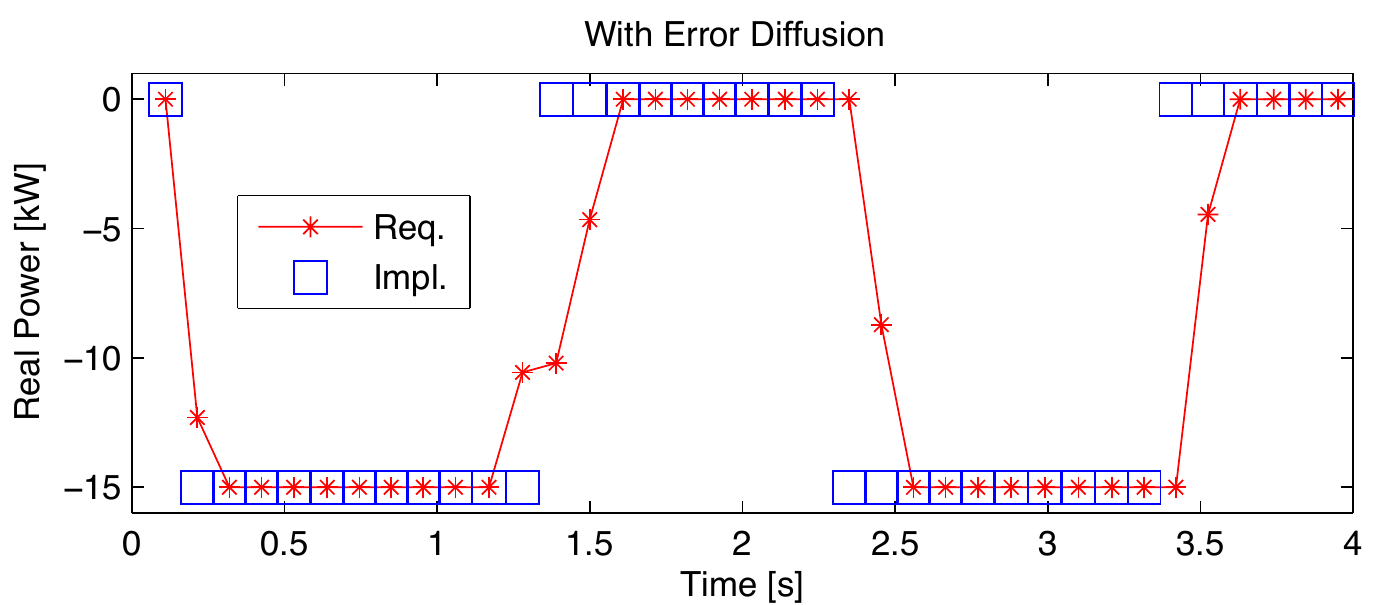}
\caption{Sequence of requested vs. implemented setpoints belonging to the heater agent. In the top figure (no error diffusion) the grid agent ``gets stuck'' on a setpoint that is not close to the optimal value. In the bottom figure (with error diffusion), the heater switches on and off with an appropriate duty cycle (the switching frequency is limited by the locking-duration parameter). }
\label{fig:errorDiff_heater}
\end{figure}

\begin{figure}
\centering
\includegraphics[width=.8\columnwidth]{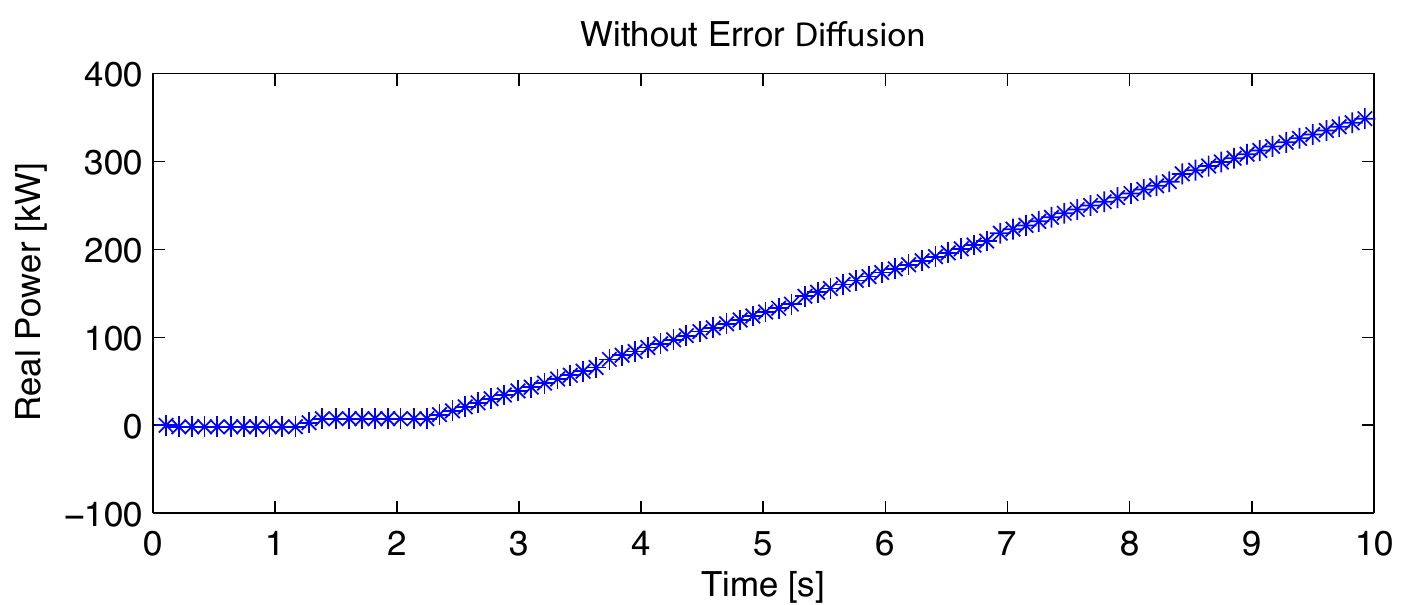}\\
\includegraphics[width=.8\columnwidth]{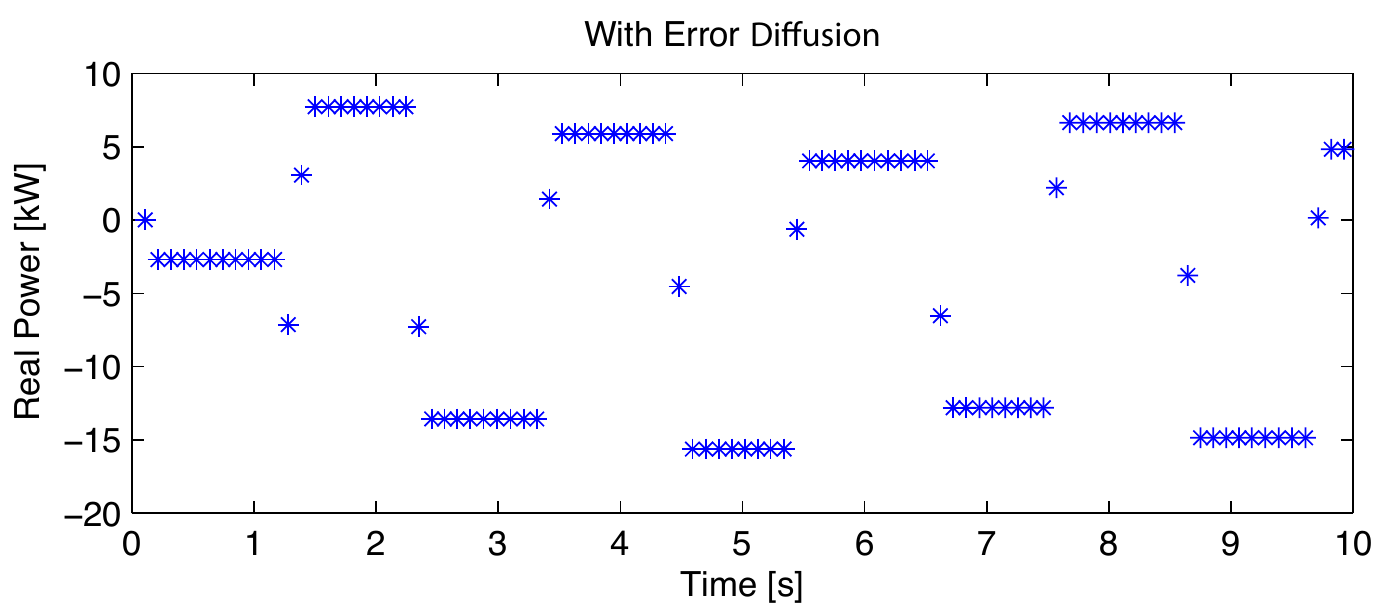}
\caption{Plots of the accumulated error of the heater agent. In the absence of error diffusion (top figure), the accumulated error grows linearly with time. With error diffusion (bottom figure), the accumulated error is bounded from above and below.}
\label{fig:accerror_heater}
\end{figure}

\begin{figure}
\centering
\includegraphics[width=.8\columnwidth]{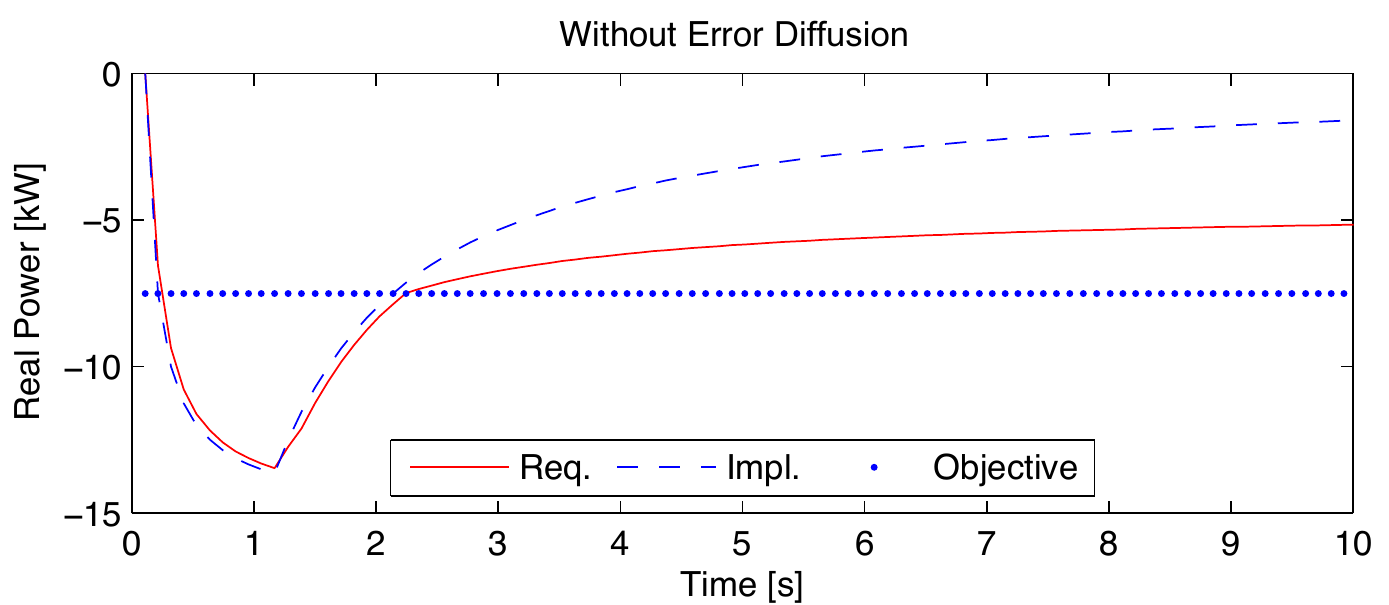}\\
\includegraphics[width=.8\columnwidth]{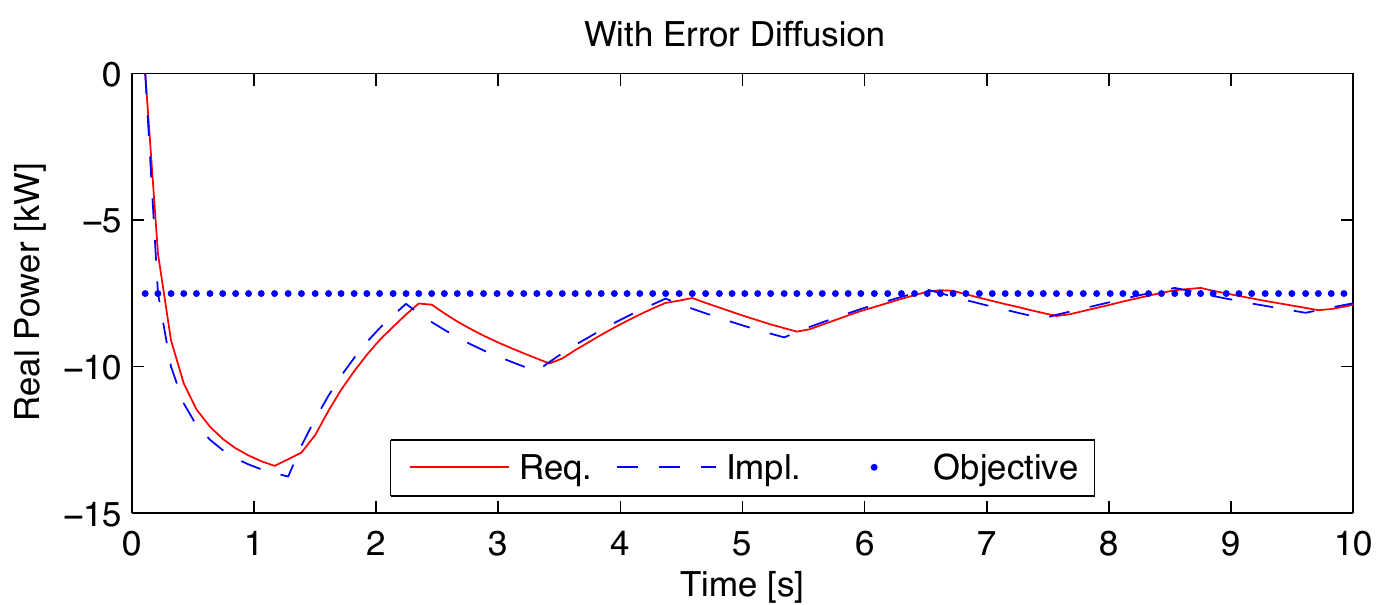}
\caption{Time-averaged sequences of requested and implemented setpoints belonging to the heater agent. In the absence of error diffusion (top figure), the implemented setpoint does not converge to the objective, where the objective is the minimizer of the cost function. With error diffusion (bottom figure), the sequence of implemented setpoint converges towards the objective.}
\label{fig:convergence}
\end{figure}

\begin{figure}
\centering
\includegraphics[width=.8\columnwidth]{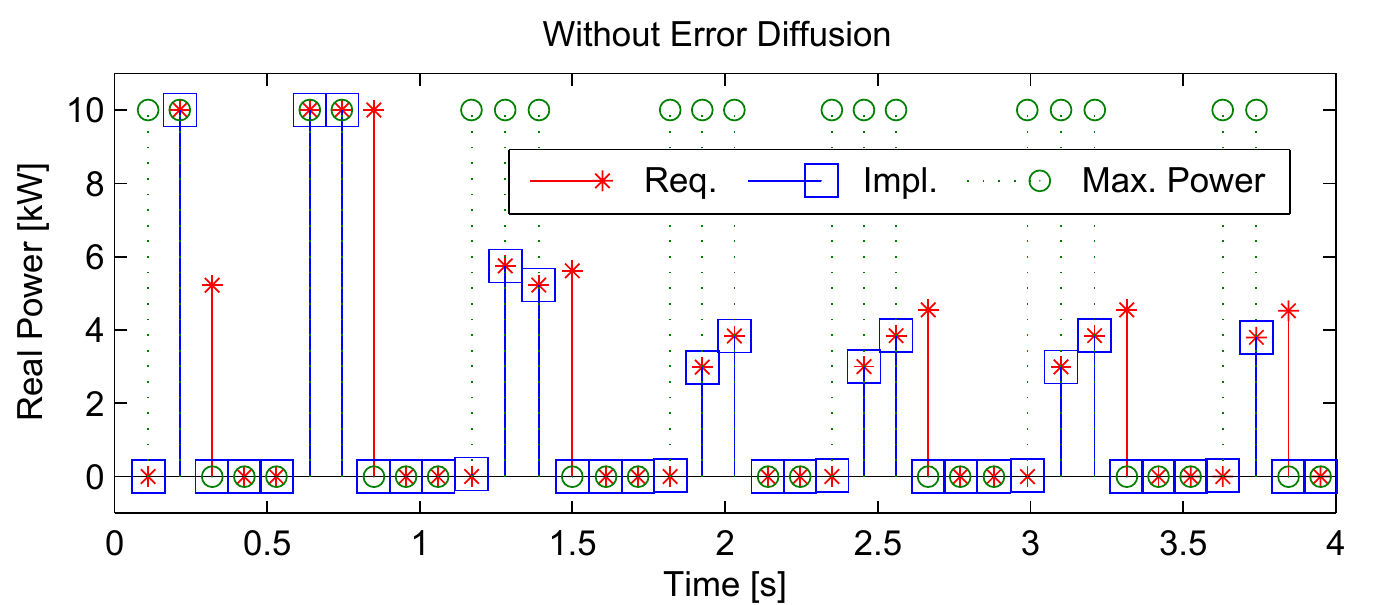}
\includegraphics[width=.8\columnwidth]{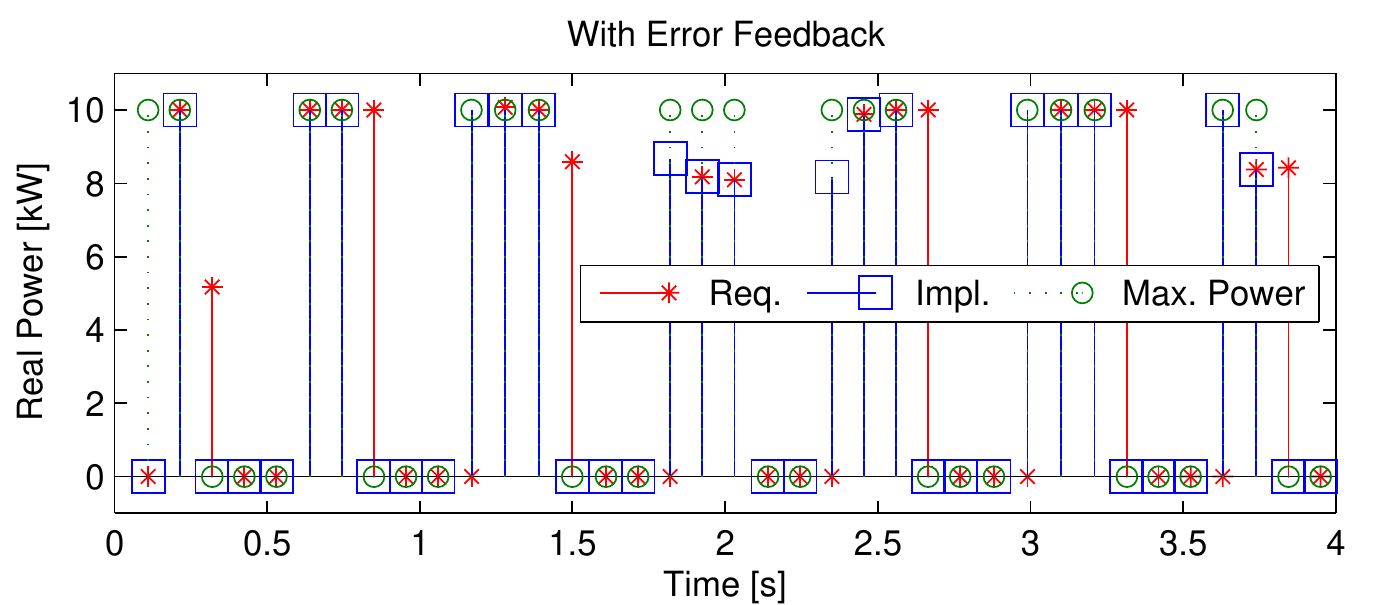}
\caption{Sequence of requested vs. implemented setpoints belonging to the PV agent. The bottom figure shows that error diffusion helps to maximize utilization of the PV.}
\label{fig:utilization}
\end{figure}

\begin{figure}
\centering
\includegraphics[width=.8\columnwidth]{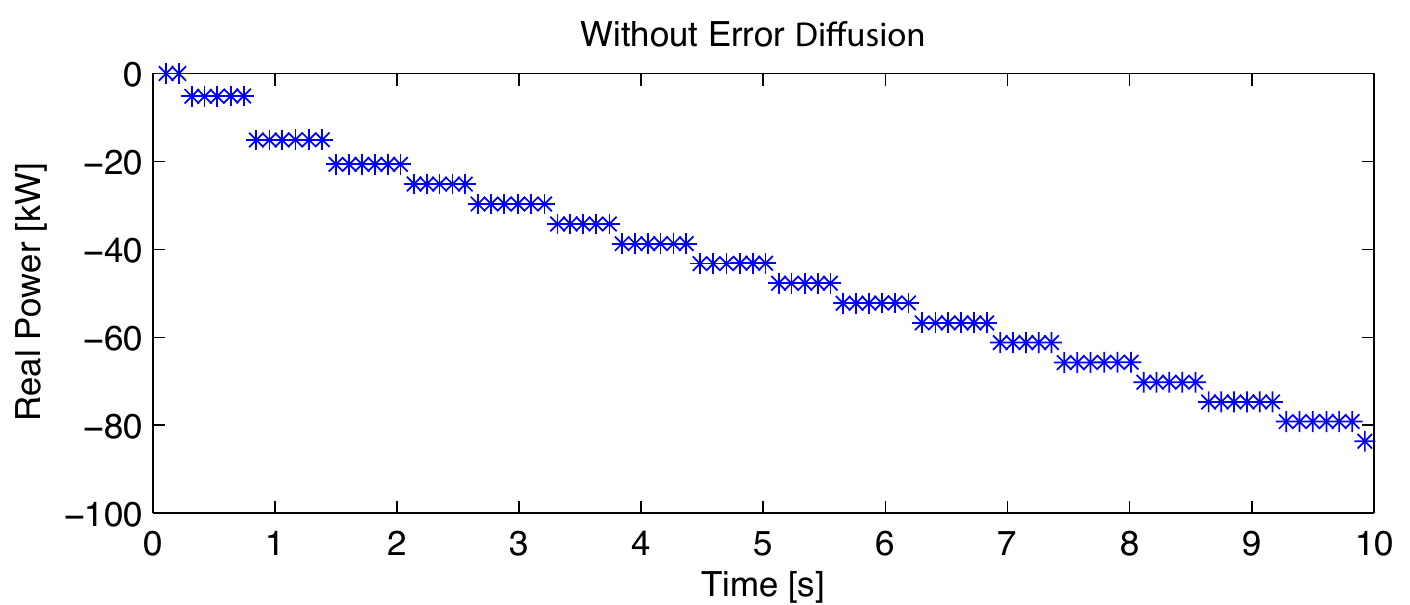}
\includegraphics[width=.8\columnwidth]{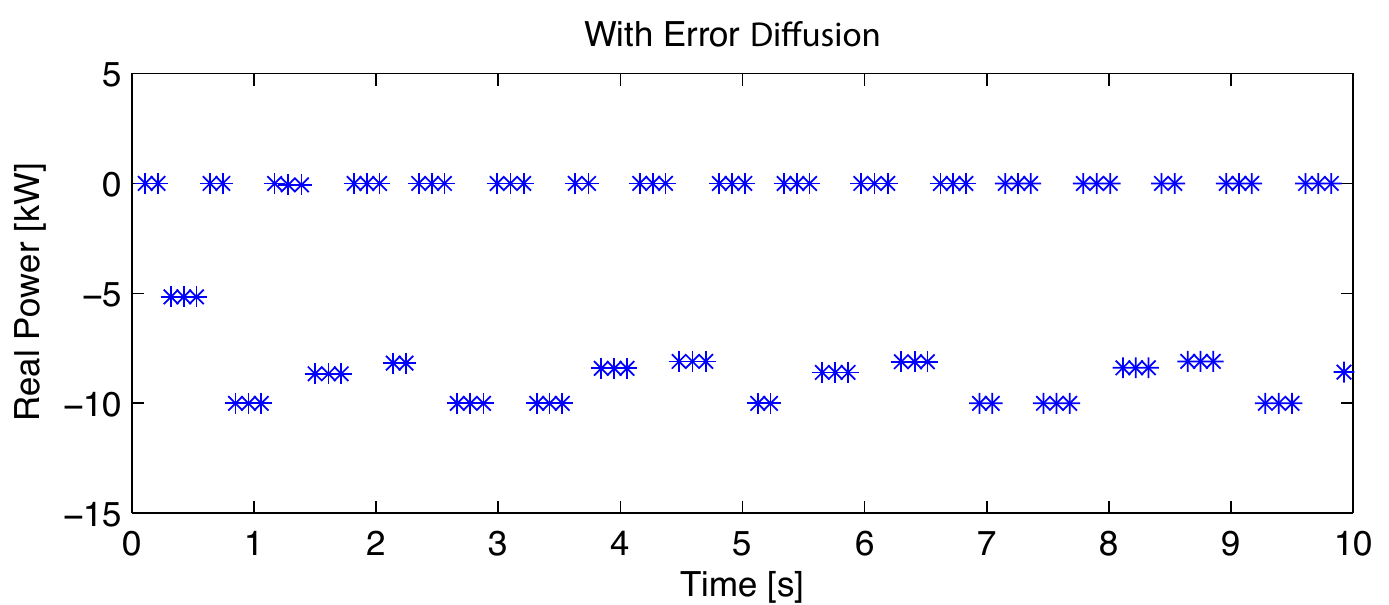}
\caption{Plot of the accumulated error of the PV agent. Like in the heater-agent case, the accumulated error grows unboundedly in the absence of error diffusion (top figure), whereas the accumulated error is bounded with error diffusion.}
\label{fig:accerror_PV}
\end{figure}

\begin{figure}
\centering
\includegraphics[width=.8\columnwidth]{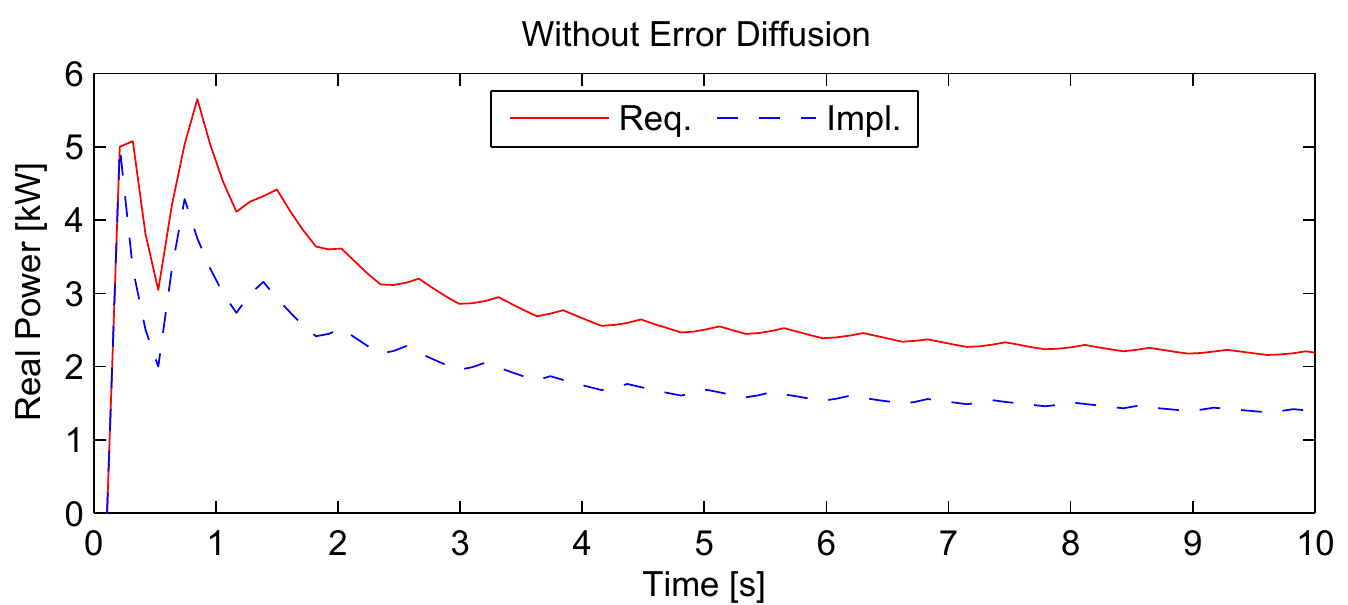}
\includegraphics[width=.8\columnwidth]{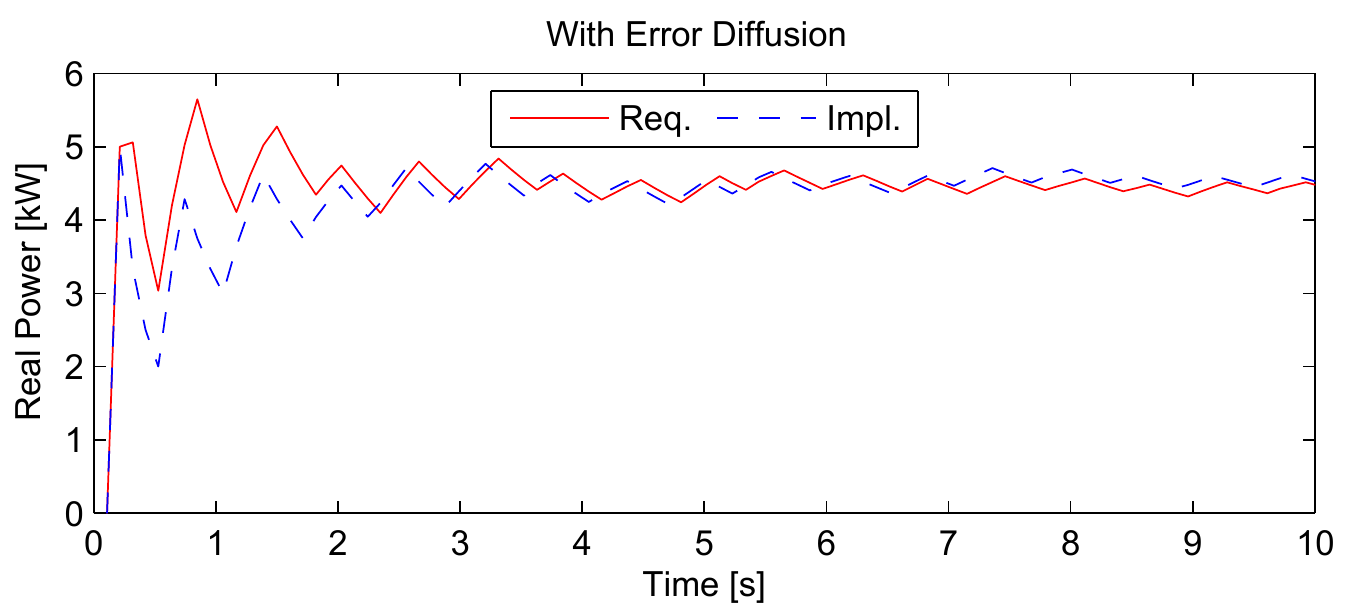}
\caption{Time-averaged sequence of requested vs. implemented setpoints belonging to the PV agent. Also this plot shows that error diffusion helps to increase the  utilization of the PV.}
\label{fig:errorDiffPV}
\end{figure}

\section{Discussion}

In this section, we comment and emphasize several important issues related to our results.

We first discuss the following conjecture related to Theorem \ref{theo:min_G} and \ref{theo:bound_G}.
\paragraph*{Conjecture} Consider a case where the collection $\mathbb{S}$ contains a single point set \mcal{S}, which is shown in \reffig{reg2}. One vertex, $p \in \mcal{S}$, 
lies inside the convex hull of the four other points in \mcal{S}, whose coordinates are $(\pm 10, \pm 10)$. 
When $p$ is moved towards the boundary of $\conv{\mcal{S}}$,  it is easy to see that the corresponding (bounded) Voronoi cell (shown in dark gray) can become arbitrarily large. (Note that if $p$ is placed on the boundary, its Voronoi cell becomes an unbounded set).
Interestingly, we observe that the minimal invariant set (light gray) also becomes larger as $p$ is moved towards $\partial \conv{\mcal{S}}$, in particular, 
it exactly covers the Voronoi cell (when the latter is shifted by $-p$) in the cases corresponding to the three positions of $p$ shown in \reffig{reg2}.
This leads us to conjecture that 
an invariant region must cover all bounded Voronoi cells, when shifted by their corresponding vertex.

%

\begin{figure}
\includegraphics[width=.3\columnwidth]{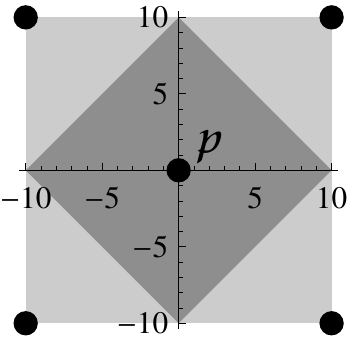}
\hfill 
\includegraphics[width=.25\columnwidth]{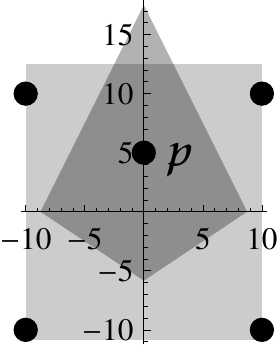}
\hfill 
\includegraphics[width=.2\columnwidth]{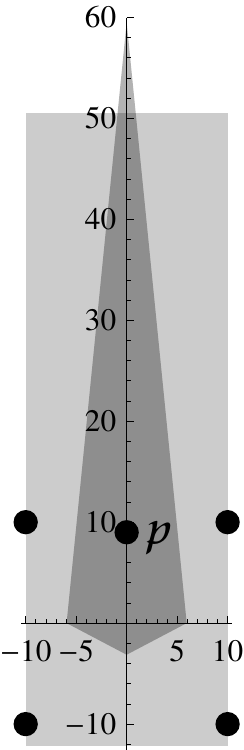}
\caption{A case where the bounded Voronoi cell (in dark gray), corresponding to a point $p$ in the interior of the convex hull of all points ($\conv{\mcal{S}}$), can grow arbitrarily large as $p$ is moved towards the boundary of $\conv{\mcal{S}}$. 
The corner points of \mcal{S} are at coordinates $(\pm 10, \pm 10)$.  
The three plots correspond to three positions of $p$: $p_\mathrm{x}=0$ and $p_\mathrm{y} \in \{0,5,9\}$, where $p_\mathrm{x}$ and $p_\mathrm{y}$ represent the x- and y-coordinate of $p$. In each plot, the minimal invariant set corresponding to \mcal{S} is shown in light gray, and exactly covers the bounded Voronoi cell that corresponds to $p$ (when the latter is translated by $-p$) in these examples; we conjecture that the minimal invariant set must cover all bounded Voronoi cells (when translated appropriately).} \label{fig:reg2}
\end{figure}

This conjecture is also corroborated by condition (3) in Theorem \ref{theo:bound_G}. Indeed, in order to proof boundedness of the minimal invariant set, we require that all bounded Voronoi cells are \emph{uniformly} bounded.

Finally, we would like to point out that our result for the case of the PV system presented in \refsec{pvexample} can be generalized to any local controller with collection $\SS$ having the following properties:
\begin{itemize}
  \item Every $\S \in \SS$ is convex.
  \item The sets are monotonic in the sense that for any $\S, \S' \in \SS$, either $\S \subseteq \S'$ or $\S' \subseteq \S$. Thus, there exists $\S_{\max} \in \SS$ such that $\S \subseteq \S_{\max}$ for all $\S \in \SS$.
  \item For every $\S \in \SS$:
  \[
  S + \bigcup_{c \in S}(S_{\max})_{\{c\}} - c \subseteq \S_{\max},
  \]
  or in other words, for every $x \in S$ and $y \in \S_{\max}$
  \[
  x + y - \vor[\S]{y} \in \S_{\max}.
  \]
\end{itemize}
Under these conditions, it is easy to show that the set $\S_{\max}$ is the minimal $F$-invariant set with respect to $\SS$. The proof is identical to that of \refthm{pv}.

\section{Proofs} \label{sec:proofs}
The proofs in this section extend the previous results to our more general case. In particular, the proof of minimality  is an extension of the proofs given in \cite{nowicki2004}, while the proof of boundedness is an extension of proofs in \cite{adler2005, tresser2007}.

\subsection{Preliminaries}
We next list well-known results that are useful in our proofs.

\begin{proposition}[Convex hull is non-decreasing] \label{lem:conv_mono}
For all sets $A$ and $B$ with $A \subseteq B$,
\[
\conv{X} \subseteq \conv{Y}
\]
\end{proposition}

\begin{corollary}
For all sets $A$ and $B$,
\[
\conv{(A \cup B)} \supseteq \conv{A} \cup \conv{B}
\]
\end{corollary}

\begin{proposition}[The Minkowski sum is distributive over unions] \label{prop:mink_union}
Let $G$ be a group and let $A$, $B$ and $C$ be arbitrary subsets of $G$. Then, 
\[
A + (B \cup C) = (A + B) \cup (A + C).
\]
\end{proposition}

\begin{proposition}[Minkowski sums and convex hulls commute]
For arbitrary sets A and B,
\[
\conv{(A + B)} = \conv A + \conv B.
\]
\end{proposition}

\subsection{Proofs for the Case of Perfect Prediction} \label{sec:proof_perf}
In this section, we present the proofs for the results of \refsec{perf}. 
For clarity of exposition,  we first present proofs for the case of single set $\S$, and then extend to a collection $\SS$. 

\subsubsection{Single-Set Case}

Recall \refdef{g} of the operator $\mathfrak{g}_\S$.
\begin{lemma} \label{lem:A_in_g}
For any $A \subseteq \reals^d$, we have that
\[
A \subseteq \mathfrak{g}_\S(A).
\]
\label{lem:Asubg}
\end{lemma}

\begin{proof}
Let $v \in A$. There exists $c \in \S$ such that $c = \vor{c + v}$. In particular, $c$ is one of the vertices of $\conv{\S}$. Therefore, $v = (c + v) - c = x - \vor{x}$ for $x = c + v \in (\P + A)_{\{c\}}$ and $\vor{x} = c$, implying that $v \in \mathfrak{g}_\S(A)$.
\end{proof}

\begin{proposition} \label{prop:inv}
$A \subseteq \reals^d$ is $G$-invariant with respect to $\S$ if and only if 
\[
A = \mathfrak{g}_\S(A)
\]
\end{proposition}

\begin{proof}
Note that by Lemma \ref{lem:A_in_g}, we only need to consider the condition $A \supseteq \mathfrak{g}_\S(A)$ for the purpose of the proof.

$(\Rightarrow)$ Assume that $A \supseteq \mathfrak{g}_\S(A)$. Let $x \in \P + A$. Also, by the definition of Voronoi cells, $x \in V_\S(\vor{x})$. Thus,
\[
x \in (\P + A)_{\{\vor{x}\}}
\]
and
\[
x - \vor{x} \in (\P + A)_{\{\vor{x}\}} - \vor{x} \subseteq \mathfrak{g}_\S(A) \subseteq A,
\]
where the last set inclusion follows by the hypothesis. 
Hence, by Definition \ref{def:invariance}, $A$ is an invariant set.

$(\Leftarrow)$ Assume that $A$ is an invariant set. Also, assume by the way of contradiction that $A \not\supseteq \mathfrak{g}_\S(A)$. Namely, there exists $v \in \mathfrak{g}_\S(A)$ such that $v \notin A$. But every $v \in \mathfrak{g}_\S(A)$ can be written as $v = x - \vor{x}$ for some $x \in \P + A$. In other words, there exists $x \in  \P + A$ such that $x - \vor{x} \notin A$, a contradiction to Definition \ref{def:invariance}.
\end{proof}

\begin{lemma}[Properties of $\mathfrak{g}_\S$] \label{lem:g_prop}
\begin{enumerate}
    \item[] 
    \item[(i)] (Monotonicity) If $A \subseteq B$ then $\mathfrak{g}_\S(A) \subseteq \mathfrak{g}_\S(B)$.
    \item[(ii)] (Additivity) $\mathfrak{g}_\S(A \cup B) = \mathfrak{g}_\S(A) \cup \mathfrak{g}_\S(B)$.
\end{enumerate}
\end{lemma}

\begin{proof}
Property $(i)$ is straightforward. To prove (ii), first note that by the monotonicity property (i), we have that $\mathfrak{g}_\S(A) \subseteq \mathfrak{g}_\S(A \cup B)$ and $\mathfrak{g}_\S(B) \subseteq \mathfrak{g}_\S(A \cup B)$, hence $\mathfrak{g}_\S(A) \cup \mathfrak{g}_\S(B) \subseteq \mathfrak{g}_\S(A \cup B)$. For the other direction, let $v \in \mathfrak{g}_\S(A \cup B)$. Then, $v = x - \vor{x}$ for some $x \in \P + A \cup B = (\P + A) \cup (\P + B)$, where the last equality follows by \refprop{mink_union}. Therefore, $v \in \mathfrak{g}_\S(A)$ or $v \in \mathfrak{g}_\S(B)$. This implies that $\mathfrak{g}_\S(A) \cup \mathfrak{g}_\S(B) \supseteq \mathfrak{g}_\S(A \cup B)$ and completes the proof of the Lemma.
\end{proof}

\begin{definition}[Iterates of $\mathfrak{g}_\S(A)$] \label{def:g_iter}
In the setting of Definition \ref{def:g}, we let
\[
\mathfrak{g}^0_\S(A) := A
\]
and
\[
\mathfrak{g}^n_\S(A) := \mathfrak{g}_\S(\mathfrak{g}^{n-1}_\S(A)) \quad \text{for } n\in \natnum, n \geq 1.
\]
\end{definition}

\begin{proposition}[Monotonicity of the iterates] \label{prop:mono}
Let $A \subseteq \reals^d$. Then the iterates $\mathfrak{g}^n_\S(A)$ are monotonic, in the sense that $\mathfrak{g}^n_\S(A) \subseteq \mathfrak{g}^{n'}_\S(A)$ for all $n \leq n'$. Thus, there exists
\[
\mathfrak{g}^\infty_\S(A) := \lim_{n \rightarrow \infty} \mathfrak{g}^n_\S(A) = \bigcup_{n\geq0} \mathfrak{g}^n_\S(A).
\]
\end{proposition}

\begin{proof}
This proposition is a direct consequence of Lemma \ref{lem:g_prop} (i) and Lemma \ref{lem:A_in_g}.  
\end{proof}

\begin{theorem} \label{theo:main_single}
Let $A \subseteq \reals^d$. The set $\mathfrak{g}^\infty_\S(A)$ is the \emph{minimal} $G$-invariant set containing the set $A$. 
\end{theorem}

\begin{proof}
We first prove the invariance of $\mathfrak{g}^\infty_\S(A)$. We have that
\begin{eqnarray*}
\mathfrak{g}_\S( \mathfrak{g}^\infty_\S(A)) &=& \mathfrak{g}_\S\left( \bigcup_{n \geq 0}  \mathfrak{g}^n_\S(A) \right) \\
&=& \bigcup_{n \geq 0} \mathfrak{g}_\S\left(\mathfrak{g}^n_\S(A) \right) \\
&=& \bigcup_{n \geq 1} \mathfrak{g}^n_\S(A) \\
&=& A \cup \bigcup_{n \geq 1} \mathfrak{g}^n_\S(A) = \mathfrak{g}^\infty_\S(A),
\end{eqnarray*}
where the second equality follows by Lemma \ref{lem:g_prop} (ii). Thus, by Proposition \ref{prop:inv},  $\mathfrak{g}^\infty_\S(A)$ is invariant. It also contains $A$ by its definition.

To prove minimality, let $Q$ be any invariant set containing $A$. Then
\[
Q \supseteq \mathfrak{g}_\S( Q ) \supseteq \mathfrak{g}_\S( A ),
\]
where the first inclusion follows by Proposition \ref{prop:inv} and the second inclusion follows by the monotonicity property (Lemma \ref{lem:g_prop} (i)). Applying these rules recursively, we obtain that
\[
Q \supseteq \mathfrak{g}_\S^n( A ), \quad n \geq 0.
\]
This implies that
\[
Q \supseteq \bigcup_{n=0}^N \mathfrak{g}^n_\S(A), \quad N \geq 0
\]
and hence
\[
Q \supseteq \mathfrak{g}^\infty_\S(A).
\]
This completes the proof of the Theorem.
\end{proof}

\subsubsection{Convex Case} \label{sec:proof_convex}
We next prove a variant of Theorem \ref{theo:main_single} for the case of convex minimal $G$-invariant sets.

\begin{definition}[Iterates of $\mathfrak{G}_\S(A)$] \label{def:g_iter_conv}
In the setting of Definition \ref{def:g_conv}, we let
\[
\mathfrak{G}^0_\S(A) := A
\]
and
\[
\mathfrak{G}^n_\S(A) := \mathfrak{G}_\S(\mathfrak{G}^{n-1}_\S(A)) \quad \text{for } n\in \natnum, n \geq 1.
\]
\end{definition}

\begin{proposition} \label{prop:inv_conv}
A \emph{convex} set $A \subseteq \reals^d$ is invariant if and only if 
\[
A = \mathfrak{G}_\S(A)
\]
\end{proposition}

\begin{proof}
Observe that by Lemma \ref{lem:A_in_g}, we have for any $A$ that
\[
A \subseteq \mathfrak{g}_\S(A) \subseteq \mathfrak{G}_\S(A).
\]
Thus, we next focus on the condition $A \supseteq \mathfrak{G}_\S(A)$.

$(\Rightarrow)$ Assume that $A \supseteq \mathfrak{G}_\S(A)$. Then clearly $A \supseteq \mathfrak{g}_\S(A)$, and by Proposition \ref{prop:inv}, $A$ is an invariant set.

$(\Leftarrow)$ Assume that $A$ is a convex invariant set. Also, assume by the way of contradiction that $A \not\supseteq \mathfrak{G}_\S(A)$. Namely, there exists $v \in \mathfrak{G}_\S(A)$ such that $v \notin A$. But every $v \in \mathfrak{G}_\S(A)$ can be written as $v = \sum_i \beta_i v_i$ with $v_i \in \mathfrak{g}_\S(A)$ and $\sum_i \beta_i = 1$, $\beta \geq 0$. By the invariance of $A$, we thus have that $v_i \in A$, but $v = \sum_i \beta_i v_i \notin A$, a contradiction to convexity of $A$.
\end{proof}

\begin{lemma}[Monotonicity of $\mathfrak{G}_\S$] \label{lem:g_prop_conv}
If $A \subseteq B$ then $\mathfrak{G}_\S(A) \subseteq \mathfrak{G}_\S(B)$.
\end{lemma}

\begin{proof}
The result follows by Lemma \ref{lem:g_prop} (i) and the monotonicity of the convex hull (Lemma \ref{lem:conv_mono}).
\end{proof}

\begin{proposition}[Monotonicity of the iterates] \label{prop:mono_conv}
Let $A \subseteq \reals^d$ be a convex set. Then the iterates $\mathfrak{G}^n_\S(A)$ are monotonic, in the sense that $\mathfrak{G}^n_\S(A) \subseteq \mathfrak{G}^{n'}_\S(A)$ for all $n \leq n'$. Thus, there exists
\[
\mathfrak{G}^\infty_\S(A) := \lim_{n \rightarrow \infty} \mathfrak{G}^n_\S(A) = \bigcup_{n\geq0} \mathfrak{G}^n_\S(A).
\]
\end{proposition}

\begin{proof}
By Lemma \ref{lem:A_in_g}, we have that $A \subseteq \mathfrak{G}_\S(A)$. The result then follows by applying the monotonicity property of $\mathfrak{G}_\S$ (Lemma \ref{lem:g_prop_conv}) for this inclusion recursively, and using Definition \ref{def:g_iter}.
\end{proof}

\begin{theorem} \label{theo:main_single_conv}
Let $A \subseteq \reals^d$ be a convex set. The set $\mathfrak{G}^\infty_\S(A)$ is the \emph{minimal convex} invariant set containing the set $A$. 
\end{theorem}

\begin{proof}
The proof is similar to that of Theorem \ref{theo:main_single}. The invariance follows by
\begin{eqnarray*}
\mathfrak{g}_\S( \mathfrak{G}^\infty_\S(A)) &=& \mathfrak{g}_\S\left( \bigcup_{n \geq 0}  \mathfrak{G}^n_\S(A) \right) \\
&=& \bigcup_{n \geq 0} \mathfrak{g}_\S\left(\mathfrak{G}^n_\S(A) \right) \\
&\subseteq& \bigcup_{n \geq 0} \mathfrak{G}_\S\left(\mathfrak{G}^n_\S(A) \right) \\
&=&\bigcup_{n \geq 1} \mathfrak{G}^n_\S(A) \\
&=& A \cup \bigcup_{n \geq 1} \mathfrak{G}^n_\S(A) = \mathfrak{G}^\infty_\S(A).
\end{eqnarray*}

To prove minimality, let $Q$ be any invariant convex set containing $A$. Then
\[
Q = \mathfrak{G}_\S( Q ) \supseteq \mathfrak{G}_\S( A ),
\]
where the first inclusion follows by Proposition \ref{prop:inv_conv} and the second inclusion follows by the monotonicity property (Lemma \ref{lem:g_prop_conv}). Applying these rules recursively, we obtain that
\[
Q \supseteq \mathfrak{G}_\S^n( A ), \quad n \geq 0.
\]
This implies that
\[
Q \supseteq \bigcup_{n=0}^N \mathfrak{G}^n_\S(A), \quad N \geq 0
\]
and hence
\[
Q \supseteq \mathfrak{G}^\infty_\S(A).
\]

This completes the proof of the Theorem.
\end{proof}

\subsubsection{Multiple-Set Case} \label{sec:proof_mult_g}

Recall \refdef{g} of the operator $\mathfrak{g}_\SS$. We next give the proofs of Proposition \ref{prop:inv_mult} and Theorem \ref{theo:min_G} presented in \refsec{perf}. We also prove the convex counterpart of Theorem \ref{theo:min_G}, namely Theorem \ref{theo:min_G_conv} (i). The proof of boundedness of the minimal invariant sets (namely, proofs of Theorem \ref{theo:bound_G} and Theorem \ref{theo:min_G_conv} (ii)) is given in \refsec{proof_bounded}.

\begin{proof}[Proof of Proposition \ref{prop:inv_mult}]
By Lemma \ref{lem:A_in_g}, $A \subseteq \mathfrak{g}_{\S} (A)$ for every $\S \in \SS$. 
Hence, $A \subseteq \mathfrak{g}_\SS (A)$ and part (i) of the proposition follows.

We next prove the two directions of part (ii).

$(\Rightarrow)$ Assume $A \supseteq \mathfrak{g}_\SS (A)$. Then, also $A \supseteq \mathfrak{g}_{\S} (A)$ for all $\S \in \SS$. 
Thus, by Theorem \ref{theo:main_single}, $A$ is invariant for all $\S \in \SS$, 
hence jointly invariant.

$(\Leftarrow)$ Assume $A$ is invariant for all $\S \in \SS$. 
Then $A \supseteq \mathfrak{g}_{\S} (A)$ all $\S \in \SS$, 
therefore also $A \supseteq \bigcup_{\S \in \SS } \mathfrak{g}_{\S} (A) = \mathfrak{g}_\SS (A)$.
\end{proof}

To prove Theorem \ref{theo:min_G}, we will need the following auxiliary results.
\begin{lemma}[Properties of $\mathfrak{g}_\SS$] \label{lem:g_prop_mult}
\begin{enumerate}
    \item[] 
    \item[(i)] (Monotonicity) If $A \subseteq B$ then $\mathfrak{g}_\SS (A) \subseteq \mathfrak{g}_\SS (B)$.
    \item[(ii)] (Additivity) $\mathfrak{g}_\SS (A \cup B) = \mathfrak{g}_\SS (A) \cup \mathfrak{g}_\SS (B)$.
\end{enumerate}
    
\end{lemma}

\begin{proof}
This result follows trivially by Lemma \ref{lem:g_prop} and by using the definition of $\mathfrak{g}_\SS$ as the union of $\mathfrak{g}_{\S}$ over $\S \in \SS$.
\end{proof}

\begin{proposition}[Monotonicity of the iterates] \label{prop:mono_mult}
Let $A \subseteq \reals^d$. Then the iterates $\mathfrak{g}^n_\SS (A)$ are monotonic, in the sense that $\mathfrak{g}^n_\SS (A) \subseteq \mathfrak{g}^{n'}_\SS (A)$ for all $n \leq n'$. Thus, there exists
\[
\mathfrak{g}^\infty_\SS (A) := \lim_{n \rightarrow \infty} \mathfrak{g}^n_\SS (A) = \bigcup_{n\geq0} \mathfrak{g}^n_\SS (A).
\]
\end{proposition}

\begin{proof}
This proposition is a direct consequence of Lemma \ref{lem:g_prop_mult} (i) and \refprop{inv_mult} (i).  
\end{proof}

\begin{proof}[Proof of Theorem \ref{theo:min_G}]
The proof is exactly the same as that of Theorem \ref{theo:main_single} by using the properties of $\mathfrak{g}^\infty_\SS$ established in Lemma \ref{lem:g_prop_mult} and Proposition \ref{prop:mono_mult}.
\end{proof}

Finally, the proof of Theorem \ref{theo:min_G_conv} (i) is exactly the same as that of Theorem \ref{theo:main_single_conv} given the results for the operator $\mathfrak{g}_\SS (A)$ proved above.

\subsubsection{Proof of Boundedness of the Minimal $G$-Invariant Set} \label{sec:proof_bounded}
In this section, we prove Theorem \ref{theo:bound_G} and Theorem \ref{theo:min_G_conv} (ii). To that end, we use the results from \cite{adler2005, tresser2007}, where it was proven that there exists a convex bounded invariant set that contains the origin in the special case where a set $\S$ is the collection of \emph{corner points} (vertexes) of $\conv \S$. 

For any $\S \in \SS$, consider a partition of $\S$ imposed by $\conv S$ into two disjoint sets: 
\begin{enumerate}
    \item $\S_c$: the ``corner points'' of $\S$, namely $\S_c \subseteq \S$ such that $\S_c \subseteq \partial \, \conv{S}$.
    \item $\S_{nc}$: the ``inner points'' of $\S$, namely $\S_{nc} \subseteq \S$ such that $\S_{nc} \not\subseteq \partial \, \conv{S}$.
\end{enumerate}
Note that in this definition, we include in the corner points also any points in $\S$ that lie on the edges of $\conv{\S}$. It is easy to see that the Voronoi cells associated with $\S_{nc}$ are bounded, while the Voronoi cells associated with $\S_{c}$ are unbounded. Moreover, $\conv{\S} = \conv{\S_c}$.

We have the following basic result about the Voronoi cells of $\S$ and $\S_c$.

\begin{lemma} \label{lem:vor_Sc}
We have that $V_\S(c) \subseteq V_{\S_c}(c)$ for all $c \in \S_c$.
\end{lemma}

\begin{proof}
Let $c \in \S_c$ and $x \in V_\S(c)$. By the definition of the Voronoi cell
\[
\|x - c\| \leq \|x - c'\|, \quad \forall c' \in \S.
\]
However, since $\S_c \subseteq \S$, it is also true that
\[
\|x - c\| \leq \|x - c'\|, \quad \forall c' \in \S_c,
\]
implying that $x \in V_{\S_c}(c)$.
\end{proof}

In \cite{adler2005, tresser2007}, it was shown that for the collection $\SS_c = \{\S_c, \S \in \SS\}$, under the hypothesis of Theorem \ref{theo:bound_G},  for any $r \geq 0$ large enough, one can construct a \emph{bounded convex} invariant set $Q_r$ which contains a ball with radius $r$ centered in the origin. That is
\begin{equation} \label{eqn:inv_Sc}
\forall \S_c \in \SS_c, \quad \mathfrak{g}_{\S_c}(Q_r) = Q_r.
\end{equation}
We next extend this proof to the case where $\S_{nc} \neq \emptyset$.

By the definition of $\mathfrak{g}$, we have that 
\begin{eqnarray*}
&& \mathfrak{g}_{\S}(Q_r) =  \bigcup_{c \in \S } (\P+Q_r) \cap V_{\S}(c) -c \\
&=& \left [ \bigcup_{c \in \S_c } (\P+Q_r) \cap V_{\S}(c) -c \right ] \bigcup \\
&& \hspace{2cm} \left [ \bigcup_{c \in \S_{nc} } (\P+Q_r) \cap V_{\S}(c) -c \right ] \\
&\subseteq& \left [ \bigcup_{c \in \S_c } (\P+Q_r) \cap V_{\S_c}(c) -c \right ] \bigcup  \\
&& \hspace{2cm}\left [ \bigcup_{c \in \S_{nc} } (\P+Q_r) \cap V_{\S}(c) -c \right ] \\
&=& \mathfrak{g}_{\S_c}(Q_r) \bigcup \left [ \bigcup_{c \in \S_{nc} } (\P+Q_r) \cap V_{\S}(c) -c \right ], 
\end{eqnarray*}
where the set inclusion follows by \reflem{vor_Sc}. 
Now observe that there exists $r_0 \geq 0$ such that for all $r \geq r_0$,
\[
(\P+Q_r) \cap V_{\S}(c)  = V_{\S}(c), \quad c \in \S_{nc}
\]
as the Voronoi cells for $c \in \S_{nc}$ are bounded. Thus, for $r \geq r_0$, 
\[
\bigcup_{c \in \S_{nc} } (\P+Q_r) \cap V_{\S}(c) -c
\]
is a bounded set. Moreover, since the Voronoi cells for $c \in \S_{c}$ are unbounded, 
\[
\lim_{r \rightarrow \infty} \bigcup_{c \in \S_c } (\P+B(0, r)) \cap V_{\S_c}(c) -c = \reals^d.
\]
Hence, there exists $r_1 \geq r_0$, such that for all $r \geq r_1$
\begin{align}
\bigcup_{c \in \S_{nc} }& (\P+Q_r) \cap V_{\S}(c) -c \notag \\
&\subseteq \bigcup_{c \in \S_c } (\P+B(0, r)) \cap V_{\S_c}(c) -c \notag \\
&\subseteq \bigcup_{c \in \S_c } (\P+Q_r) \cap V_{\S_c}(c) -c \notag \\
&= \mathfrak{g}_{\S_c}(Q_r) \label{eqn:unbounded_domin}. 
\end{align}

By the uniform boundedness of $V_\S(c)$ for $c\in \S_{nc}$ and $\S \in \SS$, it follows that there exists $r_1^* < \infty$ such that \eqref{eqn:unbounded_domin} holds for all $r \geq r_1^*$ and all $\S \in \SS$.
Thus, for all $\S \in \SS$, 
 $\mathfrak{g}_{\S}(Q_r) \subseteq \mathfrak{g}_{\S_c}(Q_r) = Q_r$, where the last equality follows by the invariance of $Q_r$ for $\SS_c$ \eqref{eqn:inv_Sc}. Therefore,  for $r \geq r_1^*$,
\[
\mathfrak{g}_{\SS}(Q_r) = \bigcup_{\S \in \SS} \mathfrak{g}_{\S}(Q_r) \subseteq Q_r,
\]
implying that $Q_r$ is a $G$-invariant set with respect to $\SS$ by \refprop{inv_mult}. 
Since $Q_r$ is a bounded convex set, the results of both Theorem \ref{theo:bound_G} and Theorem \ref{theo:min_G_conv} (ii)  follow.

\subsection{Proofs for the Case of Persistent Prediction}

In this section, we present the proofs for the results of \refsec{persist}.

\subsubsection{Single-Set Case} \label{sec:proof_single_p}
We note that the proofs for minimality of $F$-invariant sets for the case of a single set $\S$ (namely, involving the operator $\mathfrak{p}_\S$ in \refdef{p}) are given in \cite{nowicki2004} for the case where a set $\S$ is the collection of \emph{corner points} (vertexes) of $\conv \S$. This proof extends exactly in the same way as that of $G$-invariance provided in \refsec{proof_perf}, thus we present the next results without a proof.


\begin{lemma}[Properties of $\mathfrak{p}_\S$; Extension of Lem. 3.11 in \cite{nowicki2004}] \label{lem:p_prop}
\begin{enumerate}
    \item[] 
    \item[(i)] (Monotonicity) If $A \subseteq B$ then $\mathfrak{p}_\S(A) \subseteq \mathfrak{p}_\S(B)$.
    \item[(ii)] (Additivity) $\mathfrak{p}_\S(A \cup B) = \mathfrak{p}_\S(A) \cup \mathfrak{p}_\S(B)$.
\end{enumerate}
\end{lemma}

\begin{proposition}[Extension of Lemma 3.2 and Proposition 3.12 in \cite{nowicki2004}] \label{prop:inv_p}
Any $D \subseteq \reals^d$ satisfies $D \subseteq \mathfrak{p}_\S(D)$. Moreover, 
$D$ is $F$-invariant with respect to $\S$ if and only if 
$
\mathfrak{p}_\S(D) \subseteq D.
$
\end{proposition}

\begin{proposition}[Monotonicity of the iterates; Extension of Lemma 3.13 in \cite{nowicki2004}] \label{prop:mono_p}
Let $D \subseteq \reals^d$. Then the iterates $\mathfrak{p}^n_\S(D)$ are monotonic, in the sense that $\mathfrak{p}^n_\S(D) \subseteq \mathfrak{p}^{n'}_\S(D)$ for all $n \leq n'$. Thus, there exists
\[
\mathfrak{p}^\infty_\S(D) := \lim_{n \rightarrow \infty} \mathfrak{p}^n_\S(D) = \bigcup_{n\geq0} \mathfrak{p}^n_\S(D).
\]
\end{proposition}

\begin{theorem}[Extension of Corollary 3.15 in \cite{nowicki2004}] \label{theo:main_single_p}
Let $D \subseteq \reals^d$. The set $\mathfrak{p}^\infty_\S(D)$ is the \emph{minimal} $F$-invariant set containing the set $D$. 
\end{theorem}

\subsubsection{Convex Case} \label{sec:proof_conv_p}
The case of the convex iteration (namely, the one involving the operator $\mathfrak{P}_\S$ in \refdef{p}) is not treated in \cite{nowicki2004}. Hence, we next present the proofs for the following results, which are similar to the ones given in \refsec{proof_convex}.

\begin{proposition} \label{prop:inv_conv_p}
A \emph{convex} set $D \subseteq \reals^d$ is $F$-invariant if and only if 
\[
D = \mathfrak{P}_\S(D)
\]
\end{proposition}

\begin{proof}
Observe that by \refprop{inv_p}, we have for any $D$ that
\[
D \subseteq \mathfrak{p}_\S(D) \subseteq \mathfrak{P}_\S(D).
\]
Thus, we next focus on the condition $D \supseteq \mathfrak{P}_\S(D)$.

$(\Rightarrow)$ Assume that $D \supseteq \mathfrak{P}_\S(D)$. Then clearly $D \supseteq \mathfrak{p}_\S(D)$, and by Proposition \ref{prop:inv_p}, $D$ is an $F$-invariant set.

$(\Leftarrow)$ Assume that $D$ is a convex $F$-invariant set. Also, assume by the way of contradiction that $D \not\supseteq \mathfrak{P}_\S(D)$. Namely, there exists $v \in \mathfrak{P}_\S(D)$ such that $v \notin D$. But every $v \in \mathfrak{P}_\S(D)$ can be written as $v = \sum_i \beta_i v_i$ with $\sum_i \beta_i = 1$, $\beta \geq 0$, and $v_i \in \mathfrak{p}_\S(D) \subseteq D$ by the $F$-invariance of $D$. Therefore, we have that $v_i \in D$, but $v = \sum_i \beta_i v_i \notin D$, a contradiction to convexity of $D$.
\end{proof}

\begin{lemma}[Monotonicity of $\mathfrak{P}_\S$] \label{lem:p_prop_conv}
If $A \subseteq B$ then $\mathfrak{P}_\S(A) \subseteq \mathfrak{P}_\S(B)$.
\end{lemma}

\begin{proof}
The result follows by Lemma \ref{lem:p_prop} (i) and the monotonicity of the convex hull (Lemma \ref{lem:conv_mono}).
\end{proof}

\begin{proposition}[Monotonicity of the iterates] \label{prop:mono_conv_p}
Let $D \subseteq \reals^d$ be a convex set. Then the iterates $\mathfrak{P}^n_\S(D)$ are monotonic, in the sense that $\mathfrak{P}^n_\S(D) \subseteq \mathfrak{P}^{n'}_\S(D)$ for all $n \leq n'$. Thus, there exists
\[
\mathfrak{P}^\infty_\S(D) := \lim_{n \rightarrow \infty} \mathfrak{P}^n_\S(D) = \bigcup_{n\geq0} \mathfrak{P}^n_\S(D).
\]
\end{proposition}

\begin{proof}
By \refprop{inv_p}, we have that $D \subseteq \mathfrak{P}_\S(D)$. The result then follows by applying the monotonicity property of $\mathfrak{P}_\S$ (Lemma \ref{lem:p_prop_conv}) for this inclusion recursively.
\end{proof}

\begin{theorem} \label{theo:main_single_conv_p}
Let $D \subseteq \reals^d$ be a convex set. The set $\mathfrak{P}^\infty_\S(D)$ is the \emph{minimal convex} $F$-invariant set containing the set $D$. 
\end{theorem}

\begin{proof}
The proof is similar to that of Theorem \ref{theo:main_single_conv}. The invariance follows by
\begin{eqnarray*}
\mathfrak{p}_\S( \mathfrak{P}^\infty_\S(D)) &=& \mathfrak{p}_\S\left( \bigcup_{n \geq 0}  \mathfrak{P}^n_\S(D) \right) \\
&=& \bigcup_{n \geq 0} \mathfrak{p}_\S\left(\mathfrak{P}^n_\S(D) \right) \\
&\subseteq& \bigcup_{n \geq 0} \mathfrak{P}_\S\left(\mathfrak{P}^n_\S(D) \right) \\
&=&\bigcup_{n \geq 1} \mathfrak{P}^n_\S(D) \\
&=& D \cup \bigcup_{n \geq 1} \mathfrak{P}^n_\S(D) = \mathfrak{P}^\infty_\S(D).
\end{eqnarray*}

To prove minimality, let $Q$ be any invariant convex set containing $D$. Then
\[
Q \supseteq \mathfrak{P}_\S( Q ) \supseteq \mathfrak{P}_\S( D ),
\]
where the first inclusion follows by Proposition \ref{prop:inv_conv_p} and the second inclusion follows by the monotonicity property (Lemma \ref{lem:p_prop_conv}). Applying these rules recursively, we obtain that
\[
Q \supseteq \mathfrak{P}_\S^n( D ), \quad n \geq 0.
\]
This implies that
\[
Q \supseteq \bigcup_{n=0}^N \mathfrak{P}^n_\S(D), \quad N \geq 0
\]
and hence
\[
Q \supseteq \mathfrak{P}^\infty_\S(D).
\]
This completes the proof of the Theorem.
\end{proof}

\subsubsection{Multiple-Set Case}

Recall \refdef{p} of the operators $\mathfrak{p}_\SS$ and $\mathfrak{P}_\SS$. 
Also, recall the proofs for the case of $G$-invariance given in \refsec{proof_mult_g}. Observe that, given the results of Sections \ref{sec:proof_single_p} and \ref{sec:proof_conv_p} for a single-set case, the proof of \refprop{inv_mult_p} is exactly the same as that of \refprop{inv_mult}. Also, the proof of \refthm{main_uncertain} (i) is exactly the same as these of  Theorem \ref{theo:min_G} and Theorem \ref{theo:min_G_conv} (i).

\subsubsection{Proof of Boundedness of the Minimal $F$-Invariant Set} \label{sec:proof_bounded_p}
The proof of \refthm{main_uncertain} (ii) follows exactly that of the boundedness of the minimal $G$-invariant set provided in full detail in \refsec{proof_bounded}. 
It is omitted here to due to space constraints.
\ifCLASSOPTIONcaptionsoff
  \newpage
\fi



\bibliographystyle{IEEEtran}
\bibliography{IEEEabrv,refs}
%

%

%




\end{document}